\documentclass[reqno, 11pt, a4paper]{amsart}

\usepackage{latexsym,ifthen,xspace}
\usepackage{amsmath,amssymb,amsthm}
\usepackage{enumerate, calc}
\usepackage{verbatim}
\usepackage[colorlinks=true, pdfstartview=FitV, linkcolor=blue,
  citecolor=blue, urlcolor=blue]{hyperref}
\usepackage[text={33pc,605pt},centering]{geometry}
\usepackage{graphicx}
\graphicspath{{./images/}}

\newtheorem{theorem}{Theorem}[section]
\newtheorem{lemma}[theorem]{Lemma}
\newtheorem{prop}[theorem]{Proposition}
\newtheorem{assumption}[theorem]{Assumption}
\newtheorem{corro}[theorem]{Corollary}

\theoremstyle{definition}
\newtheorem{definition}[theorem]{Definition}

\theoremstyle{remark}
\newtheorem{remark}[theorem]{Remark}

\numberwithin{equation}{section}

\newcommand{\oW}{\ensuremath{\overline{\cW}}}
\newcommand{\oE}{\ensuremath{\overline{\cE}}}
\newcommand{\cB}{\ensuremath{\mathcal B}}
\newcommand{\cC}{\ensuremath{\mathcal C}}
\newcommand{\cD}{\ensuremath{\mathcal D}}
\newcommand{\cE}{\ensuremath{\mathcal E}}
\newcommand{\cF}{\ensuremath{\mathcal F}}
\newcommand{\cG}{\ensuremath{\mathcal G}}
\newcommand{\cH}{\ensuremath{\mathcal H}}

\newcommand{\cK}{\ensuremath{\mathcal K}}
\newcommand{\cL}{\ensuremath{\mathcal L}}
\newcommand{\cM}{\ensuremath{\mathcal M}}
\newcommand{\cN}{\ensuremath{\mathcal N}}
\newcommand{\cO}{\ensuremath{\mathcal O}}

\newcommand{\cW}{\ensuremath{\mathcal W}}

\newcommand{\bbN}{\ensuremath{\mathbb N}}

\newcommand{\bbR}{\ensuremath{\mathbb R}}

\newcommand{\bbZ}{\ensuremath{\mathbb Z}}
\newcommand{\bfM}{\ensuremath{\mathbf M}}

\newcommand{\bsM}{\ensuremath{\boldsymbol M}}

\newcommand{\diver}{\ensuremath{\mathrm{div}}}
\newcommand{\ldef}{\ensuremath{\mathrel{\mathop:}=}}
\newcommand{\rdef}{\ensuremath{=\mathrel{\mathop:}}}
\newcommand{\supp}{\mathrm{supp}}

\DeclareMathOperator{\capa}{\mathrm{Cap}}
\DeclareMathOperator{\loc}{\mathrm{loc}}

\begin{document}

\title[Weighted Sobolev spaces and diffusion processes: an example]{Singular weighted Sobolev spaces and diffusion processes: an example
 (due to V.V. Zhikov)}
%   author one information

\author{Alberto Chiarini}
\address{ETH Zurich}
\curraddr{101, R\"amistrasse\\ 8048, Zurich, Switzerland}
\email{alberto.chiarini@math.ethz.ch}
\thanks{}

%    author two information

\author{Pierre Mathieu}
\address{Aix-Marseille Universit\'e,}
\curraddr{CNRS, Centrale Marseille\\ I2M, UMR 7373\\ 13453, Marseille, France}
\email{pierre.mathieu@univ-amu.fr}
\thanks{}

\subjclass[2010]{}

\keywords{}

\date{\today}

\dedicatory{}

\begin{abstract} We consider the Sobolev space over $\bbR^d$ of square integrable functions whose gradient is also square integrable with respect to some positive weight.
 It is well known that smooth functions are dense in the weighted Sobolev space when the weight is uniformly bounded from below and above. This may not be the case when the weight is unbounded.
 In this paper, we focus on a class of two dimensional weights where the density of smooth functions does not hold. This class was originally introduced by V.V. Zhikov; such weights have  a unique singularity point of non-zero capacity.
 Following V.V. Zhikov, we first give a detailed analytical description of the weighted Sobolev space. Then, we explain how to use Dirichlet forms theory to associate a diffusion process to such a degenerate non-regular space.
\end{abstract}

\maketitle

\tableofcontents

\section{Introduction}
 We fix some positive weight  $\rho:\bbR^d \to [0,\infty]$ such that $\rho,\,\rho^{-1}\in L^1_{\mathrm{loc}}(\bbR^d)$. We consider the weighted Sobolev space over $\bbR^d$
 \begin{equation*}
 \cW\ldef \Big \{ u\in W^{1,1}_{\loc}(\bbR^d)\, :\,  \int_{\bbR^d} \big(|u|^2+|\nabla u|^2\big) \,\rho dx< \infty  \Big \},
\end{equation*}
where $W^{1,1}_{\loc}(\bbR^d)$ is the set of weakly differentiable functions with locally integrable gradient.
If $\rho$ is uniformly bounded from above and below, then $\cW$ coincides with the classical Sobolev space on $\bbR^d$ and it is well known that smooth functions are dense in $\cW$. In general, if the weight is not bounded,
smooth functions need not be dense in $\cW$, in which case we say that the weight is \emph{non-regular}.

Together with the space $\cW$ it is natural to consider the symmetric quadratic form
\[
\cE(u,v) := \int_{\bbR^d} \nabla u \cdot \nabla v \, d\mu, \quad u,v\in \cW,
\]
where $\mu(dx) = \rho\,dx$.
It is easy to see that the pair $(\cE,\cW)$ is a Dirichlet form on $L^2(\bbR^d,\mu)$. Accordingly, it is uniquely associated to a Markovian semigroup. This semigroup is not necessarily strongly Markov in the case  $\rho$ is non-regular. From general theory~\cite[Appendix A.4]{fukushima2011dirichlet}, there exists an equivalent regular Dirichlet form on a different state space which is associated to a strongly Markovian semigroup.
In this paper we focus on a class of non-regular weights (see~\eqref{def:weight_two} and below) with one point singularity of non zero capacity; this class was introduced by Zhikov in~\cite{zhikov1996connectedness,zhikov1998weighted}.
We are able to give a concrete description of the regularization of the corresponding Dirichlet form.
This is accomplished via a topological construction (see Section 3.3).
We will rely on the theory of many-point extensions to provide a characterization of the resolvent associated to the regularized Dirichlet form in terms of the resolvent associated to a diffusion process on $\bbR^d$ that is killed at the singular point (see Theorem~\ref{thm:semigroupW}).
The question of density of smooth functions in weighted Sobolev spaces has an intrinsic analytical flavour. With this work we show that this problem carries a probabilistic counterpart in the context of symmetric Markov processes and Dirichlet form theory.
\vspace{\baselineskip}

In order to motivate the subject of the paper from the point of view of stochastic analysis let us consider the following stochastic differential equation on $\bbR^d$
\begin{equation}\label{eq:distortedbrownianmotion}
 d X_t = \sqrt{2} d W_t + \frac{\nabla \rho}{\rho} (X_t) \,dt,\quad X_0 = x\in \bbR^d,
\end{equation}
where ${(W_t)}_{t\geq 0}$ is a $d$-dimensional Brownian motion and $\rho:\bbR^d\to [0,+\infty)$ is a function which for the moment we assume to be smooth, bounded and bounded away from zero.
The function $\rho$ is nothing else but the density with respect to the Lebesgue measure of the invariant distribution $\mu = \rho dx$ of the process $X$.
The model described in equation~\eqref{eq:distortedbrownianmotion} is known in the literature under the name of \emph{symmetric distorted Brownian Motion} or \emph{Langevin dynamics} associated to the potential $\log \rho$.

In such regular case, the classical theory of stochastic differential equations provides a unique strong solution up to the explosion time to equation~\eqref{eq:distortedbrownianmotion} whose generator is given by

\begin{equation}\label{eq:generator}
 L u \ldef  \frac{1}{\rho}\,\diver  (\rho \nabla u)  = \Delta u + \sum_{i=1}^d \partial_i (\log \rho)\partial_i u.
\end{equation}

In the literature, assumptions on the regularity and boundness of $\rho$ have been considerably softened. Particularly remarkable are the works~\cite{albeverio2003strong,rockner2015non}.
Here the authors assume that $\rho(x)>0$ for almost all $x\in \bbR^d$ and
\begin{equation}\label{eq:condrock}
 \sqrt{\rho} \in W^{1,2}_{\loc}(\bbR^d),\quad \frac{\nabla \rho}{\rho}  \in L^{(d+\epsilon)\vee 2}_{\loc}(\bbR^d,\mu),
\end{equation}
where $\epsilon>0$ and $W^{1,2}_{\loc}(\bbR^d)$ is the set of weakly differentiable functions with locally square integrable gradient.
Under this assumption, they are able to prove weak existence for any starting point in the set $\{\tilde{\rho}>0\}$,
being $\tilde{\rho}$ a continuous version of $\rho$, which exists in view of~\eqref{eq:condrock}.
Moreover, the constructed solution stays in $\{\tilde{\rho}>0\}$ before possibly going out of any ball in $\bbR^d$.
Once having shown weak existence, they can actually prove that the constructed weak solution is indeed strong and weakly, as well as pathwise, unique up to its explosion time.

For the interested reader we also mention the work~\cite{krylov2005strong} where the case of a Brownian motion with singular time dependent drift is considered. The authors prove that, in spite of the poor regularity, there is a unique strong solution up to the explosion time, provided some integrability on the drift.

If we were satisfied by just having weak solutions to equation~\eqref{eq:distortedbrownianmotion} for almost all starting points, then it would suffice to assume $\sqrt{\rho} \in W^{1,2}_{\loc}(\bbR^d)$.
Indeed, the bilinear form obtained from the operator~\eqref{eq:generator} integrating by parts $\int_{\bbR^d} u (-L) v d\mu$
\begin{equation}\label{eq:dirichlet_form}
 \cE(u,v) \ldef \int_{\bbR^d} \nabla u \cdot \nabla v \,d\mu,\quad u,v\in C_0^\infty(\bbR^d),
\end{equation}
is closable in $L^2(\bbR^d,\mu)$, where $C_0^\infty(\bbR^d)$ is the set of smooth functions with compact support. Denote by $\cH$ the completion of $C_0^\infty(\bbR^d)$ in $L^2(\bbR^d,\mu)$ with respect to the inner product $\cE_1(u,v) \ldef \cE(u,v)+\langle u, v\rangle$, where $\langle \cdot, \cdot\rangle$ is the scalar product in $L^2(\bbR^d,\mu)$.
Then, the pair $(\cE, \cH)$ is a regular strongly local Dirichlet form on $L^2(\bbR^d,\mu)$ (see Section~\ref{sec:elements} below for precise definitions). It follows from~\cite[Theorem 7.2.2]{fukushima2011dirichlet} that there exists a reversible diffusion process $(\Omega, \cM,{\{\cM_t\}}_{t\geq0},{\{X_t\}}_{t\geq0},P_x)$, $x\in \bbR^d$ in
$\bbR^d\cup \{\Delta \}$ with invariant measure $\mu$, lifetime $\zeta$ and cemetery $\Delta$. Moreover, ${\{X_t\}}_{t\geq 0}$ weakly solves~\eqref{eq:distortedbrownianmotion} for all $x\in \bbR^d$ outside a set of zero capacity.

The theory of Dirichlet forms shows its true strength when assumptions on  the regularity of $\rho$ drop and we ask for no more than measurability and some degree of integrability.
In this case it becomes clear that $L$, as defined in~\eqref{eq:generator}, must be considered only in a distributional sense, that is, one rather works directly with the bilinear form~\eqref{eq:dirichlet_form} to have a chance to construct a stochastic process formally associated to~\eqref{eq:generator}.
We require the minimal condition that
\begin{equation}\label{eq:cond}
 \rho, \rho^{-1}\in L^1_{\loc}(\bbR^d).
\end{equation}
On one hand, $\rho\in L^1_{\loc}(\bbR^d)$ is necessary to ensure that $\cE_1(u,u)<\infty$ for $u\in C_0^\infty(\bbR^d)$ and on the other hand $\rho^{-1}\in L^1_{\loc}(\bbR^d)$ grants closability of $(\cE,C_0^\infty(\bbR^d))$ in $L^2(\bbR^d,\mu)$. This assumption could be weakened even further, see for example~\cite[Chapter II]{ma2012introduction}.

Arguing as above, we can consider the completion $\cH$ of $C_0^\infty(\bbR^d)$ in $L^2(\bbR^d,\mu)$ with respect to $\cE_1$. Also in this case it turns out that $(\cE, \cH)$ is a regular strongly local Dirichlet form on $L^2(\bbR^d,\mu)$ and therefore we can find a diffusion process associated with it.

The choice of the domain of $\cE$ is extremely relevant in the characterization of the stochastic process as different domains yield different processes.
One of the most natural choices for a domain of $\cE$ is to take all the functions $u\in L^2(\bbR^d,\mu)$ for which expression~\eqref{eq:dirichlet_form} makes sense and is finite,
\begin{equation}\label{eq:W}
 \cW\ldef \Big \{ u\in W^{1,1}_{\loc}(\bbR^d)\, :\, \cE_1(u,u) < \infty  \Big \}.
\end{equation}
The class of functions $\cW$ is known in the literature as \emph{weighted} Sobolev space  with weight $\rho$. In fact, observe that the choice $\rho \equiv 1$ gives the standard Sobolev space on $\bbR^d$.
Based on~\eqref{eq:cond} and on the fact that for any ball $B\subset \bbR^d$
\begin{equation}\label{eq:localL1}
 \int_B |\nabla u| \, dx \leq {\left (\int_{\bbR^d} |\nabla u|^2 \, d\mu\right)}^\frac{1}{2}  {\left (\int_B \rho^{-1}\, dx \right)}^\frac{1}{2},
\end{equation}
one can prove that $\cW$ is an Hilbert space with inner product $\cE_1$, and consequently the pair $(\cE, \cW)$ is a closed symmetric form on $L^2(\bbR^d,\mu)$.
Noticing that for all $u\in \cW$, $v = (0\vee u)\wedge 1$ belongs to $\cW$ and $\cE_1(v,v)\leq \cE_1(u,u)$, we actually have that $(\cE,\cW)$ is a Dirichlet form on  $L^2(\bbR^d,\mu)$, which is also easily seen to be strongly local.
Clearly $\cH \subseteq \cW$ but equality is not true in general. In this article we are interested in situations where $\cH\neq \cW$. We have the following definition.

\begin{definition}
 We say that the weight $\rho:\bbR^d\to [0,+\infty)$ is \emph{regular} provided that $\cW = \cH$, that is, if smooth functions are dense in $\cW$.
\end{definition}

We will provide the interested reader with references and conditions for regularity of $\rho$ in Section~\ref{subsec:regularity} below. Assume that $\rho$ is not regular, then $\cH \subsetneq \cW$ and $(\cE,\cW)$ properly extends $(\cE,\cH)$.
A natural question is whether there is a stochastic process associated to the Dirichlet form $(\cE,\cW)$ on $L^2(\bbR^d,\mu)$. As $\cH \neq \cW$, $C_0^\infty(\bbR^d)$ is not dense in $\cW$ and we cannot be sure in general that $(\cE,\cW)$ is a regular Dirichlet form. Nonetheless, any Dirichlet form admits a regular representation on a possibly different state space (see Appendix in~\cite{fukushima2011dirichlet}).

The construction of the regular representation and of the state space is quite abstract and makes it hard to truly understand what this regularization procedure is about.
In this paper we focus on a class of examples of non-regular weights where we can construct concretely a regular representation of $(\cE, \cW)$, and consequently a diffusion process. We will be able to describe such a diffusion and compare it to that associated to $(\cE,\cH)$.

Let us give an informal description of this plan with the following example of non-regular weight. Consider
$\rho : \bbR^2\to [0,+\infty]$ defined by
\begin{equation}\label{def:weight_example}
 \rho(x)\;\ldef \;\left \{
 \begin{array}{ll}
  | x|^{-\alpha}\;, & \quad x_1 x_2 > 0,\; | x|\leq 1, \\
  | x|^\alpha \;,   & \quad x_1 x_2 < 0,\; | x|\leq 1, \\
  1 \;,             & \quad \mbox{otherwise,}
 \end{array}
 \right.
\end{equation}
with $\alpha \in (0,2)$. Let us denote by $Q_1$, $Q_2$, $Q_3$, $Q_4$ the four quadrants, starting from the top-right in anti-clockwise order.
Formally, generator~\eqref{eq:generator} is that of a Brownian motion outside $B(0,1)$, that of a Brownian motion with drift
$-\alpha x/| x|^2$ in $(Q_1\cup Q_3)\cap B(0,1)$ and of a Brownian motion with drift $\alpha x/| x |^2$ in $(Q_2\cup Q_4)\cap B(0,1)$. If we look at the radial component of this generator inside $B(0,1)$, we get
\begin{align}\label{eq:genrad1}
 L_r & = \partial_r^2 + (1-\alpha) \frac{1}{r} \partial_r\ ,\quad\mbox{on } (Q_1\cup Q_3)\cap B(0,1), \\
 \label{eq:genrad2}
 L_r & = \partial_r^2 + (1+\alpha) \frac{1}{r} \partial_r\ ,\quad\mbox{on } (Q_2\cup Q_4)\cap B(0,1).
\end{align}

The operators~\eqref{eq:genrad1} and~\eqref{eq:genrad2} are generators of two Bessel processes with parameters $(1-\alpha)/2$ and $(1+\alpha)/2$ respectively. It is well known in the literature that the condition $\alpha \in (0,2)$ implies that the process with generator~\eqref{eq:genrad1}  hits and spends no time in zero, and that the process with generator~\eqref{eq:genrad2} never hits the origin.

\begin{figure}[h]
 \centering
 \includegraphics[width=0.5\textwidth]{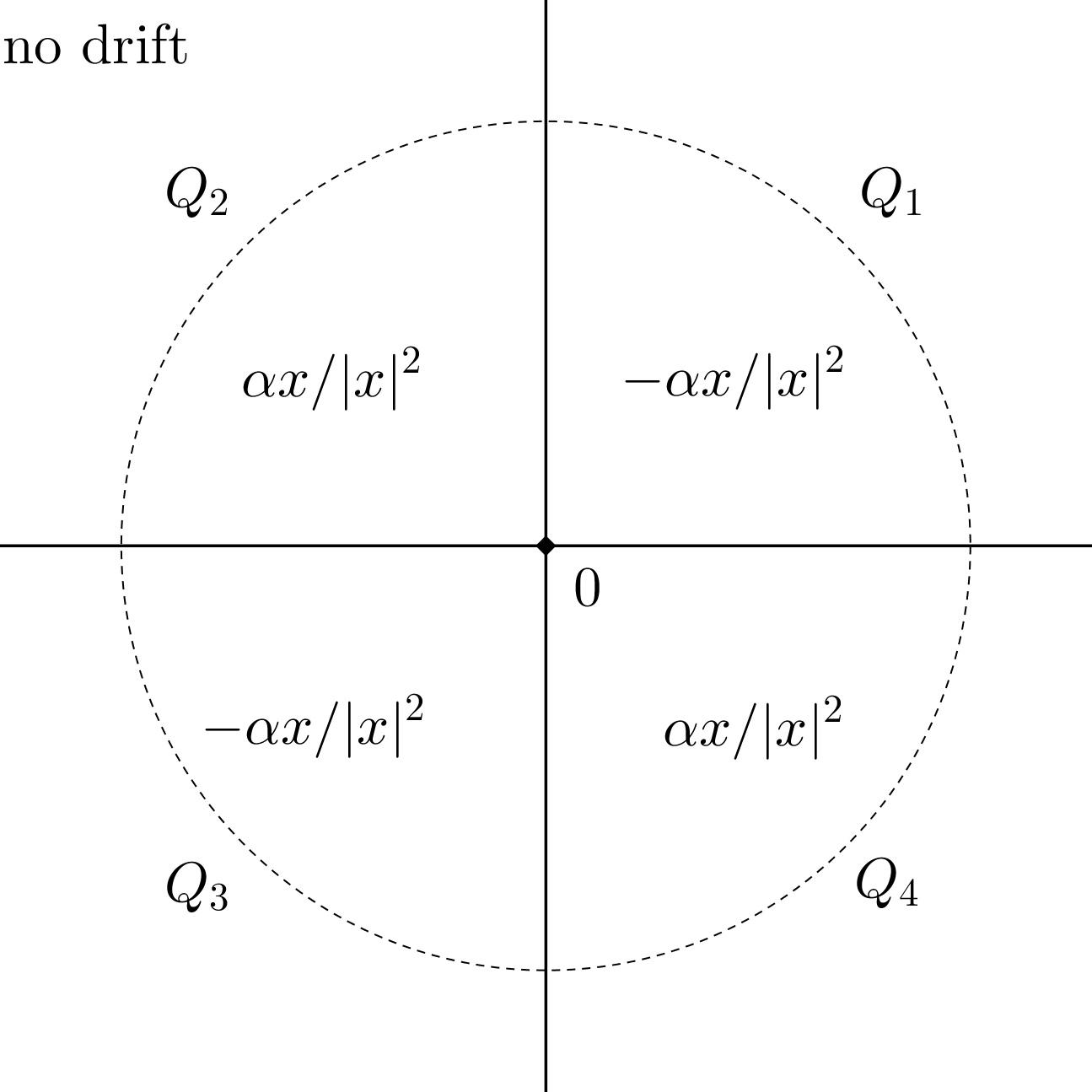}
 \caption{the drifts associated to $L$.}
\end{figure}

With this in mind, it is not hard to believe that the process associated to $(\cE,\cH)$ will eventually hit the origin but only coming from the region $Q_1\cup Q_3$. In fact, we will show that the non-regularity of $\rho$ implies that the origin has positive capacity.
By symmetry and reversibility, we can also understand that after hitting the origin the process leaves it from $Q_1$ or $Q_3$ with no preference.

A regularization of the Dirichlet form $(\cE,\cW)$ accounts in a change of the topology of $\bbR^2$ around the origin. Indeed, we will show that $\cW=\cH + \bbR\psi_0$, where $\psi_0$ as defined in~\eqref{def:psi} is continuous everywhere with the exception of zero. More precisely the function jumps from one to zero in going from $Q_1$ to $Q_3$ through the origin.

\begin{figure}[h]
 \begin{minipage}[b]{0.45\linewidth}
  \centering
  \includegraphics[width=\textwidth]{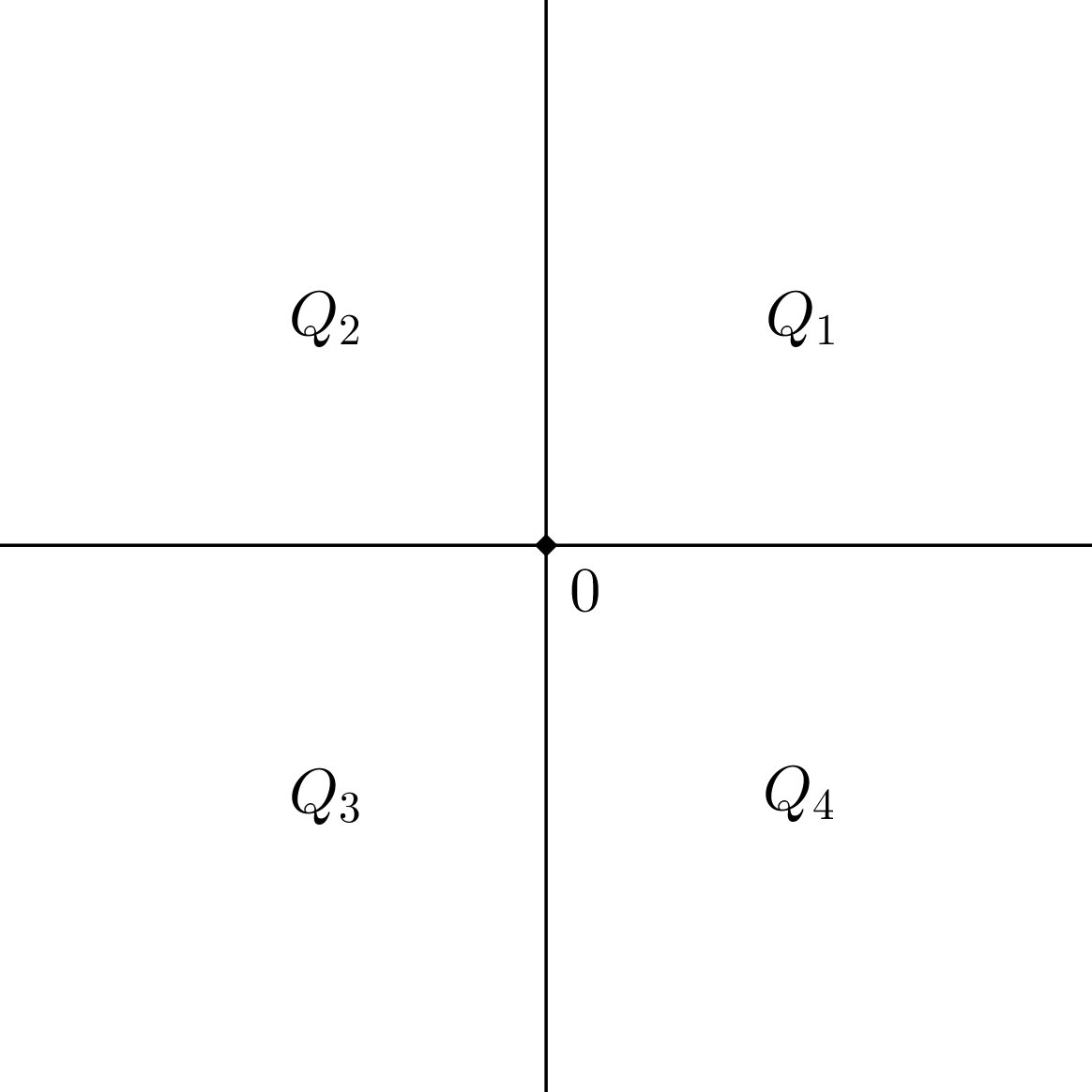}
 \end{minipage}
 \hspace{0.5cm}
 \begin{minipage}[b]{0.45\linewidth}
  \centering
  \includegraphics[width=\textwidth]{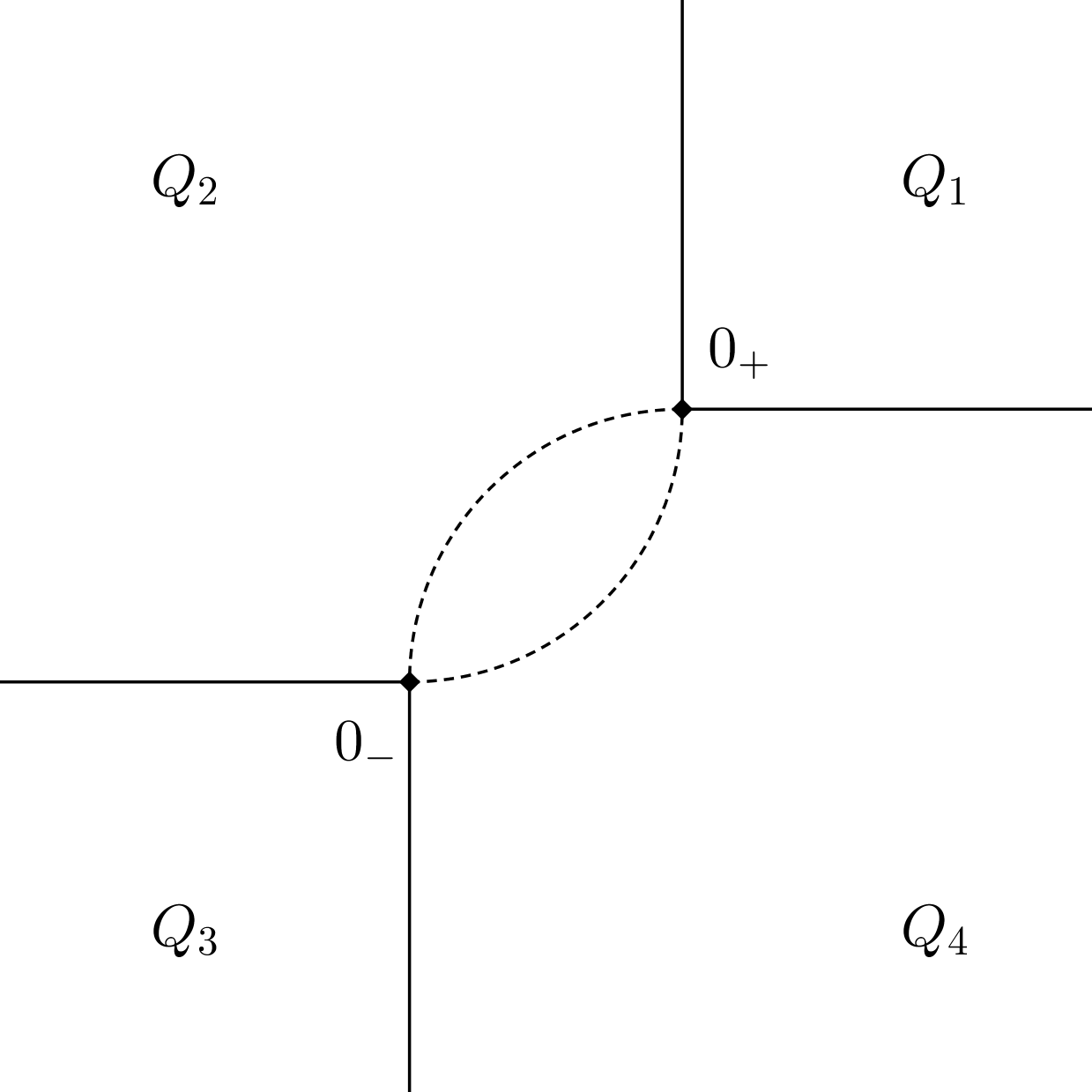}

  %\label{fig:figure2}
 \end{minipage}
 \caption{Splitting the origin in two points.}
\end{figure}

Roughly speaking, it suffices to ``split'' the origin into two points to make $\psi_0$ continuous and consequently $(\cE,\cW)$ regular. On this new state space we can now find a strong Markov process with continuous trajectories associated to $(\cE,\cW)$. Observe that now $Q_1$ and $Q_3$ are disconnected.

Since the degeneracy is found only at the origin, it is intuitively clear that the behavior of the process associated to $(\cE,\cW)$ outside zero is the same as of the process associated to $(\cE,\cH)$. In particular, it will approach the origin from $Q_1$ or $Q_3$ only.
However, when the process comes closer to the origin, the new topology, that is, the fact that $Q_1$ and $Q_3$ are disconnected, plays a role. What we observe is that after hitting the origin, due to the continuity of the sample paths, the process must leave it from the same side it came. That is, if for example the process was in $Q_1$ just before hitting zero it will be in $Q_1$ right after.
Thus, the projection of the sample trajectories of the process associated to $(\cE,\cW)$ by identifying the two points is not going to determine a strongly Markovian law on $\bbR^2$ with the Euclidean topology.
\newline

\textbf{Organization of the paper.} In Section 2 we recall basic definitions and facts about Dirichlet forms theory. We also give an overview on the subject of weighted Sobolev spaces and the question of regularity of a weight.
The most important statement of Section 2 is Proposition~\ref{prop:onepoint} which states that non-regular weights must give positive capacity to their set of degeneracy points.
In Section 3 we introduce the class of non-regular weights we are interested in.  In Section 3.1, a first analytical part will be carried out to characterize $\cW$ as a rank one extension of $\cH$. We shall then describe in Section 3.2 the Hunt process $(X,P_x)$ associated to $(\cE,\cH)$ in terms of the one obtained from $X$ by killing the sample paths upon hitting the origin.
We will proceed by performing a regularization of $(\cE,\cW)$ and by describing the Hunt process $(Y,P^\cW_x)$ associated to it in Section 3.3.
In Section 4 other examples of non-regular weights are quickly discussed.

\section{Preliminaries}

\subsection{Elements of Dirichlet forms theory}\label{sec:elements}

In this section we briefly recall some of the fundamental definitions in the theory of Dirichlet forms in the special case of locally compact separable metric spaces.
For an exhaustive treatment of the subject we refer to~\cite{chen2012symmetric, fukushima2011dirichlet} and to~\cite{ma2012introduction} for further generalizations.

Let $(E,d)$ be a locally compact separable metric space, and $\mu$ a positive Radon measure on $E$ with full support. We denote by $\langle\cdot, \cdot\rangle$  the scalar product in $L^2(E,\mu)$. Let $\cF$ be a dense linear subspace of $L^2(E,\mu)$.
A pair  $(\cE,\cF)$ is called a symmetric form on $L^2(E,\mu)$ with domain $\cF$ provided that $\cE : \cF\times \cF \to \bbR$ is bilinear and non-negative.
As customary, for $\alpha>0$ we set $\cE_\alpha (u,v)\ldef \cE(u,v) + \alpha \langle u, v \rangle$. Note that $\cF$ becomes a pre-Hilbert space with respect to the inner product $\cE_\alpha$.
In the case $\cF$  is complete with respect to the metric determined by $\cE_\alpha$, the symmetric form $(\cE,\cF)$ is said to be \emph{closed}.

Given two symmetric forms $(\cE^{(1)},\cF_1)$ and $(\cE^{(2)},\cF_2)$ on $L^2(E,\mu)$, we say that the second is an extension of the first if $\cF_1\subset \cF_2$ and $\cE^{(1)}\equiv\cE^{(2)}$  on  $\cF_1\times \cF_1$.
A symmetric form is then said to be \emph{closable} if it admits a closed extension.

\begin{definition}
 A \emph{Dirichlet form} on $L^2(E,\mu)$ is a closed symmetric form $(\cE,\cF)$ which has the additional property of being \emph{Markovian},
 that is, for each $u\in \cF$ the function $v = (0\vee u)\wedge 1$ belongs to $\cF$ and $\cE(v,v)\leq \cE(u,u)$.
\end{definition}

Let $C_0(E)$ be the set of compactly supported continuous functions on $E$. A \emph{core} of a symmetric form $(\cE,\cF)$ is by definition a subset $\cC$ of $\cF \, \cap \, C_0(E)$ such that $\cC$ is dense in $\cF$ with respect to $\cE_1$ and in $C_0(E)$ with respect to the uniform norm.

\begin{definition}
 A symmetric form $(\cE,\cF)$ on $L^2(E,\mu)$ is said to be \emph{regular} if it possesses a core.
\end{definition}

For a $\mu$-measurable function $u$ the support of the measure $u(x) \mu(dx)$ will be denoted by $\supp[u]$. Clearly, if $u\in C(E)$ then $\supp[u]$ is just the closure of the set of points $x\in E$ where $u$ is not zero.

\begin{definition} We say that a symmetric form $\cE$ is \emph{local} if for any $u,v\in\cF$ with disjoint compact support $\cE(u,v)=0$. $\cF$ is said to be \emph{strongly local} if for any $u,v\in\cF$ with compact support and such that $v$ is constant on a neighborhood of $\supp[u]$, then $\cE(u,v)=0$.
\end{definition}

The success of the theory of Dirichlet forms is due to the rich interplay between the theory of strongly continuous contraction semigroups and Markov processes, in particular symmetric Hunt processes.
A Hunt process is a Markov process that possesses useful properties such as the right continuity of sample paths,
the quasi-left-continuity and the strong Markov property (we refer to~\cite[Appendix A.2]{fukushima2011dirichlet} for definitions).
Fix a $\mu$-symmetric Hunt process $\bfM \ldef (\Omega, \cM,{\{\cM_t \}}_{t\geq0},{\{X_t\}}_{t\geq0},P_x)$ on $(E,\cB(E))$ with lifetime $\zeta$ and cemetery $\Delta$.
The transition function of  $\bfM$
\[
 p_t f(x) := E_x\left[ f(X_t);\,t<\zeta\right],\quad t\geq 0
\]
on $E$ uniquely determines a strongly continuous Markovian semigroup $T_t$, $t\geq0$ on $L^2(E,\mu)$
and with that a Dirichlet form $(\cE,\cF)$ on $L^2(E,\mu)$. In view of this, we will say that $\bfM$ is associated to $(\cE,\cF)$.

The converse turns out to be true as well, as it is better described by the following celebrated result.
\begin{theorem}[Theorem 7.2.1 in~\cite{fukushima2011dirichlet}]\label{thm:huntprocess} Let $(\cE,\cF)$ be a regular Dirichlet form on $L^2(E,\mu)$. Then there exists a $\mu$-symmetric
 Hunt process $\bsM$ on $E$ whose Dirichlet form is the given one $\cE$.
\end{theorem}

In the sequel we shall make use of the concept of capacity of a set, which we recall here below. Given a Dirichlet form $(\cE, \cF)$ on $L^2(E,m)$ and an open set $A\subset E$, consider
\[
 \cL_A :=\{u\in \cF : u\geq 1,m\mbox{-a.e.\ on } A\}.
\]
We define
\begin{equation}\label{def:capacity}
 \capa_1(A) = \left \{
 \begin{array}{ll}
  \inf_{u\in \cL_A}\cE_1(u,u), & \hbox{if $\cL_A \neq\emptyset$}, \\
  +\infty,                     & \hbox{if $\cL_A = \emptyset$.}
 \end{array}
 \right.
\end{equation}
For a general set $B\subset E$ we set
\[
 \capa_1(B) = \inf_{A\supset B, A\in \cO} \capa_1(A),
\]
where $\cO$ is the class of open sets of $E$. %We shall stress the domain with respect to which we consider the capacity by writing $\capa_1^{\cF}(\cdot)$.
The present notion tells us already that sets with zero capacity are somewhat finer than
$\mu$-negligible sets. In fact, from the very definition we get that $\mu(A)\leq \capa_1(A)$ for an open set $A$. Let $A$ be any subset of $E$.
A statement depending on $x\in A$ is said to hold \emph{quasi-everywhere}, in short q.e., if there exists a set of zero capacity $N\subset A$ such that the said statement holds for all $x\in A\setminus N$.

The concept of zero capacity sets for a regular Dirichlet form $(\cE,\cF)$ on $L^2(E,\mu)$ is better understood in connection to the exceptional sets for a Hunt process
$\bfM$ associated to $(\cE,\cF)$. A set $N\subset E$ is called \emph{exceptional} if there exists a Borel set $B\supset N$ such that $P_x(\sigma_B <\infty)=0$ for $\mu$-almost all $x\in E$,
where $\sigma_B:=\inf \{t>0\,:\,X_t\in B\}$.
A set $N\subset E$ is called \emph{properly exceptional} if $N$ is Borel,  $\mu(N)=0$, and $P_x(X_t\in N\mbox{ or }X_{t-}\in N \mbox{ for some }t\geq0)=0$ for all $x\in E\setminus N$. It can be shown that  any exceptional set is contained in a properly exceptional set.
According to~\cite[Theorem 4.2.1]{fukushima2011dirichlet} a set is exceptional if and only if it has zero capacity

We can now answer the question about uniqueness of the process associated to a Dirichlet form.
It turns out that two Hunt processes with the same Dirichlet form are \emph{equivalent}, that is, they possess a common properly exceptional set outside of which their transition functions coincide for all times~\cite[see Theorem 4.2.8]{fukushima2011dirichlet}.

Let $E_\Delta = E \cup \{\Delta \}$  be the one-point compactification of $E$. When $E$ is already compact, $\Delta$ is just an isolated point. Any function $u$ on $A\subset E$ can always be extended to a function on $A\cup \{\Delta \}$  by setting $u(\Delta) = 0$. In particular, any continuous
function on $E$ vanishing at infinity is regarded as a continuous function on $E_\Delta$.
Let $u$ be an extended real valued function defined q.e.\ on $E$. We call $u$ quasi-continuous if for each $\epsilon>0$ there exists an open set $G$ such that $\capa_1(E\setminus G) <\epsilon$ and $u|_{E\setminus G}$ is continuous.
If we require the stronger condition that $u|_{E_{\Delta}\setminus G}$ is continuous, then we say that $u$ is quasi-continuous in the restricted sense.

\begin{prop}[Theorem 4.2.2 in~\cite{fukushima2011dirichlet}]\label{continuity} If u is quasi-continuous, then there exists a properly exceptional set $N$ such that for any $x\in E\setminus N$,
 \begin{equation}\label{eq:right_cont1}
  P_x\Big(u(X_t) \text{ is right continuous and } \lim_{t\uparrow t'} u(X_t) = u(X_{t'-}),\, \forall t'\in (0,\zeta)\Big) = 1
 \end{equation}
 and
 \begin{equation}\label{eq:right_cont2}
  P_x\Big(\lim_{t\uparrow \zeta} u(X_t) = u(X_{\zeta-}), X_{\zeta-}\in E\Big) = P_x( X_{\zeta-}\in E).
 \end{equation}
\end{prop}

We recall that according to~\cite[Theorem 2.1.3]{fukushima2011dirichlet} every function $u\in \cF$ admits a quasi-continuous $\mu$-modification  $\tilde{u}$ in the restricted sense.

Fix an $\mu$-symmetric Hunt process $\bfM = (\Omega, \cM,{\{\cM_t\}}_{t\geq0},{\{X_t\}}_{t\geq0},P_x)$ on $E$ which is associated with a regular Dirichlet form $(\cE,\cF)$ on $L^2(E,\mu)$.
Given an open set $G$, one can consider the part of $\bfM$ on $G$, that is,
the stochastic process obtained from $\bfM$ by killing the sample paths upon leaving $G$, call this process $\bfM_G$.
It can be shown that $\bfM_G$ is a Hunt process on $(G,\cB(G))$ with transition probability function given by
\[
 p_t^G(x,B) \ldef P_x(X_t \in B, t < \sigma_{E\setminus G}),\quad B\in \cB(G), x\in G,
\]
where $\sigma_{E\setminus G} \ldef \inf \{t>0,\, X_t\notin G\}$.
Moreover, $\bfM_G$ is associated with a regular Dirichlet form $(\cE_G,\cF_G)$ on $L^2(G,\mu)$, where $\cE_G$ is the restriction of $\cE$ to the set $\cF_G\ldef \{ u \in \cF :\, \tilde{u}|_{E\setminus G} =0\mbox{ q.e.\ } \}$, which can be identified with a subset of $L^2(G,\mu)$ in an obvious way.

\subsection{The question of regular weights}\label{subsec:regularity} Let $\Omega$ be an open subset of $\bbR^d$. The problem of  identifying necessary and sufficient conditions such that a measurable weight $\rho:\Omega\to [0,\infty]$ is regular
has been studied quite extensively in a series of papers by Zhikov~\cite{zhikov1998weighted},~\cite{zhikov2011variational},~\cite{zhikov2013density} and~\cite{zhikov2016density}.
The author was interested in the subject mainly in relation to the theory of calculus of variations where the inequality $\cH\neq \cW$  expresses the so called Lavrentiev's phenomenon.
First we should mention that the one dimensional case does not present any interesting feature, since in this case $\cH=\cW$ for any possible  integrable weight~\cite[Section 4]{zhikov1998weighted}.

In the theory of degenerate elliptic equations the weight $\rho$ is assumed to satisfy a condition of the type
\begin{equation}\label{eq:pq_condition}
 \rho\in L^p(\Omega),\quad \rho^{-1}\in L^q(\Omega),\quad \frac{1}{p}+\frac{1}{q}<\frac{2}{d},
\end{equation}
which makes it possible to prove weighted analogues of the Sobolev embedding theorems.
The question $\cH=\cW$ was also addressed in this context but without much success, indeed, it was later found that neither condition~\eqref{eq:pq_condition} nor the condition that $\rho$ and $\rho^{-1}$ are integrable to an arbitrary power ensures the equality.
For example it is shown in~\cite{zhikov1998weighted}  for $d=2$ that the weight
\[
 \rho(x):=\left \{
 \begin{array}{ll}
  {\left(\log \frac{1}{|x|}\right)}^{-\alpha}, & x_1x_2>0, \\
  {\left(\log \frac{1}{|x|}\right)}^\alpha,    & x_1x_2<0,
 \end{array}
 \right.
\]
is non-regular for any $\alpha>1$, even though it is easily seen to be integrable to an arbitrary power.

Using the same smoothing techniques as in the classical case $\rho\equiv1$ it is possible to show that $\rho$ is regular if it belongs to the Muckenhaupt's class $A_2$, that is, if
\[
 \sup_{B}\,\left(\frac{1}{|B|}\int_{B} \rho \,dx\right) \left(\frac{1}{|B|}\int_{B}\rho^{-1}\,dx\right) < \infty,
\]
where the supremum is taken over all balls $B\subset \bbR^d$. An interesting fact was established in~\cite{cassano1994local}, namely if the Poincar\'e's inequality
\[
 \int_B \Big|u - \int_B u\rho dy\Big|^2\,\rho  dx\leq C |B|^{2/d} \int_{B} |\nabla u|^2\,\rho dx
\]
holds for all balls $B\subset\bbR^d$ and $u\in \cW$, then $\rho$ is regular. The drawback of this statement is that it is very hard to verify.

In a recent paper~\cite{zhikov2013density}, Zhikov uses a truncation argument to obtain the following more general result.
\begin{theorem} Suppose that $\rho = w w_0$ with $w_0 \in A_2$. The equality $\cH = \cW$ holds if
 \begin{equation}\label{condition:regularity}
  \lim_{n\to\infty} \frac{1}{n^2} {\left(\int_{\Omega} w^{n}w_0\,dx\right)}^{\frac{1}{n}} {\left(\int_{\Omega}w^{-n}w_0\,dx\right)}^{\frac{1}{n}} < +\infty.
 \end{equation}
\end{theorem}
As a remarkable consequence it should be noted that $\rho = w w_0$ with $w_0 \in A_2$ is regular if $\exp(t w)w_0$ and $\exp(-t w)w_0$ belong to $L^1(\Omega)$ for some $t>0$, which was a conjecture posed by De Giorgi.

A particularly interesting scenario for the question of regularity, which will also include our class of models in Section~\ref{sec:guidingmodel}, is the case where the weight is degenerate only on a compact set of measure zero.
For the following proposition, which appears in~\cite{zhikov1998weighted} we provide a sketch of the proof for the reader's convenience.

\begin{prop}\label{prop:onepoint} Let $F$ be a compact subset of $\Omega$. Assume that $\rho$ degenerates only in $F$, that is, for all $\epsilon>0$ there exists $c(\epsilon)>1$ such that
 \[
  c{(\epsilon)}^{-1} \leq \rho(x)\leq c(\epsilon),
 \]
 for almost all $x\in \Omega$ such that $d(x,F)>\epsilon$. If $F$ has zero Lebesgue measure and $\capa_1(F) = 0$, then $\rho$ is regular. Here the capacity is the one associated to $(\cE,\cH)$ on $L^2(\Omega,\mu)$.
\end{prop}

\begin{proof} We want to show that each function $u\in \cW$ can be approximated by smooth functions with respect to $\cE_1$. Without loss of generality we can assume $|u|<M$ for some positive $M$.
 Indeed, $-k\vee (u\wedge k)\to u$ with respect to $\cE_1$ by~\cite[Theorem 1.4.2]{fukushima2011dirichlet} as $k\to+\infty$.
 By definition of capacity and the compactness of $F$ we can find a sequence of functions $\phi_n\in \cH$ such that
 \[
  \int_\Omega |\nabla \phi_n|^2\rho \,dx \leq \frac{1}{n},\quad 0\leq \phi_n\leq 1,
 \]
 $\phi_n(x) = 0$ for all $x\in\Omega$ such that $d(x,F)< 1/n$ and $\phi_n \to 1$ almost-surely as $n\to\infty$, where it is important to use $\capa_1(F)=0$.
 Clearly, $u \phi_n \to u$ almost surely and therefore in $L^2(\Omega,\mu)$ by the dominated convergence theorem.

 We notice that $\supp [u\phi_n] \subset \{ x\in \Omega \,:\,d(x,F)\geq 1/n\}$ where $\rho$ is bounded from above and below.
 The density of smooth functions in the classical Sobolev space implies $u\phi_n \in \cH$.

 We now show that ${\{\nabla(u\phi_n)\}}_{n\in \bbN}$ is bounded in ${(L^2(\Omega,\mu))}^d$. By the chain rule and the triangular inequality,
 \begin{equation}\label{eq:chainrule}
  \cE(u\phi_n,u\phi_n) \leq 2 \int_{\Omega}  |\nabla u|^2 \phi_n^2\, \rho dx +  2 \int_{\Omega}  |\nabla \phi_n|^2 u^2\, \rho dx.
 \end{equation}
 The first contribution in~\eqref{eq:chainrule} is trivially uniformly bounded in $n\in\bbN$ by $2\cE(u,u)$, while the second contribution can be bounded using
 \[
  \int_{\Omega}  |\nabla \phi_n|^2 u^2\, \rho dx \leq \frac{M^2}{n}
 \]
 which is uniformly bounded as well.  It is now easy to show that $\nabla(u \phi_n)$ converges weakly to $\nabla u$ in ${(L^2(\Omega,\rho))}^d$.
 This is enough to conclude, since we proved that $\cH$ is weakly dense in $\cW$ with respect to $\cE_1$. Thus, $\cH=\cW$ being $\cH$ strongly closed.
\end{proof}

\section{The guiding model}\label{sec:guidingmodel}

We will consider a distorted Brownian motion and the corresponding Dirichlet form on the plane. We first divide $\bbR^2$ in the union of the four quadrants which meet at the origin. We name them $Q_1$, $Q_2$, $Q_3$ and $Q_4$ starting from the top-right corner in anti-clockwise order.

\begin{figure}[h]
 \centering
 \includegraphics[width=0.8\textwidth]{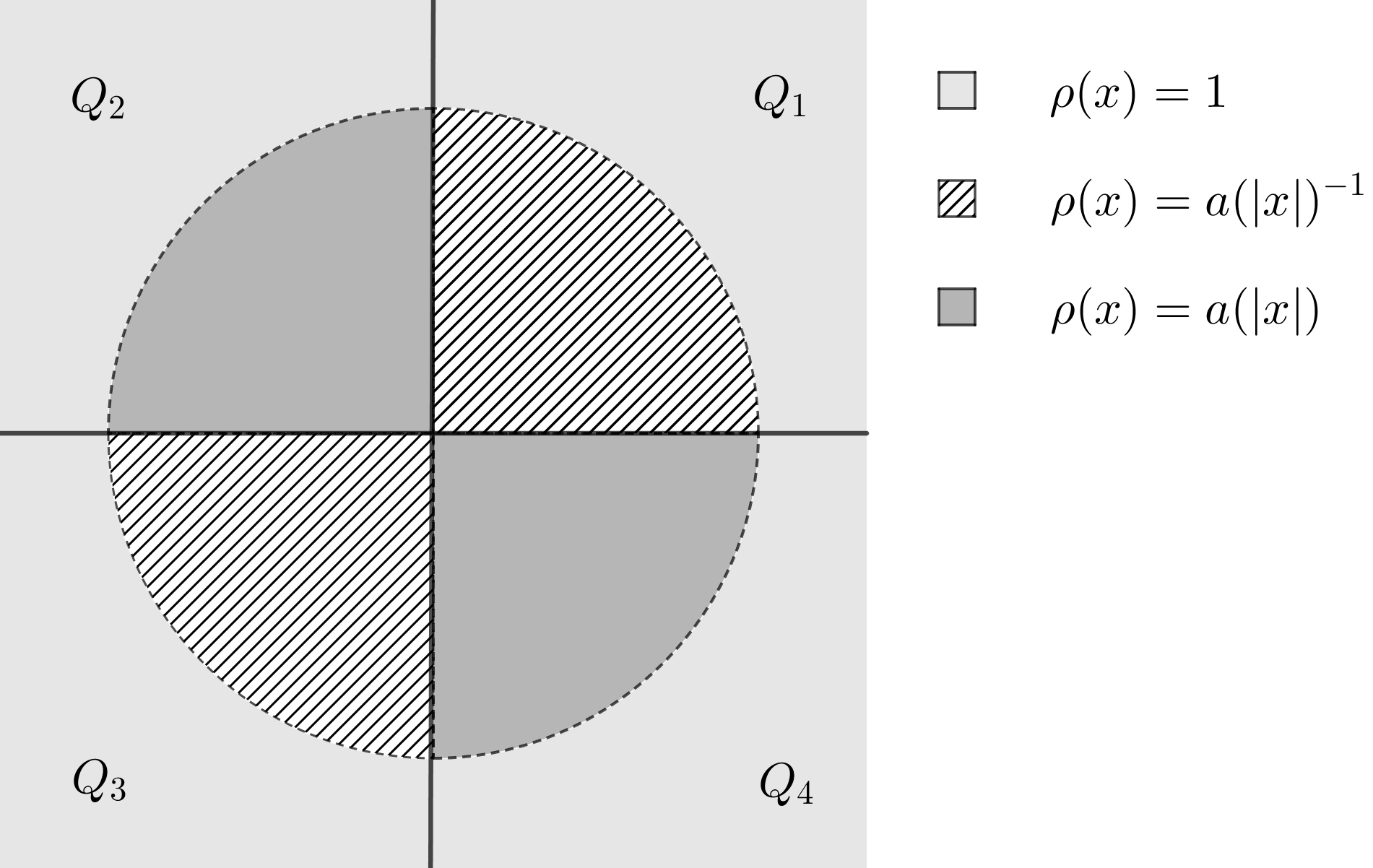}
\end{figure}

Let  $\rho : \bbR^2\to [0,+\infty]$ be the non-negative weight defined by
\begin{equation}\label{def:weight_two}
 \rho(x)\;\ldef \;\left \{
 \begin{array}{ll}
  a{(|x|)}^{-1}\;, & \quad x\in(Q_1\cup Q_3) \cap B(0,1), \\
  a(|x|)\;,        & \quad x\in(Q_2\cup Q_4) \cap B(0,1), \\
  1\;,             & \quad \mbox{otherwise.}
 \end{array}
 \right.
\end{equation}
\begin{assumption}
 The function $a: ]0,1] \to [0,\infty]$ satisfies the following conditions:
 \begin{itemize}
  \item[i)] for every $\epsilon>0$ there is a positive constant $c(\epsilon)>0$ such that
        \begin{equation}\label{ass:onepointdeg}
         a(r),\;a{(r)}^{-1}\geq c(\epsilon),\quad \forall r>\epsilon;
        \end{equation}
  \item[ii)]
        \begin{equation}\label{ass:integrability}
         \int_0^1 \frac{a(r)}{r}\,dr<\infty, \quad \int_0^1 \frac{r}{a(r)}\,dr <\infty;
        \end{equation}
  \item[iii)] there exists a constant $K>0$ such that for all $\epsilon \in (0,1)$ we have
        \begin{equation}\label{ass:onerank}
         \frac{1}{\epsilon^2}\int_0^\epsilon a^{-1}(r) r\int_0^r \frac{a(s)}{s}\,ds dr<K.
        \end{equation}
 \end{itemize}
\end{assumption}

\begin{remark} It follows immediately from~\eqref{ass:integrability}  and $\int_0^1 a(r) r\,dr \leq \int_0^1 a(r)/r\,dr$  that $\rho \in L^1_{\loc}(\bbR^2)$.
 Clearly this implies also that $\rho^{-1} \in L^1_{\loc}(\bbR^2)$ as $\rho^{-1}$ is obtained through a rotation of $\pi/2$ of $\rho$ around the origin.
\end{remark}

\begin{remark} Two functions $a:]0,1]\to \bbR$ that satisfy conditions~\eqref{ass:onepointdeg},~\eqref{ass:integrability} and~\eqref{ass:onerank} are
 $a(r) = r^{\alpha}$ with $0 < \alpha < 2$ and $a(r) = {(\log(2/r))}^{-\alpha}$ with $\alpha>1$.
 Notice that $x\mapsto |x|^{\alpha}$ with $\alpha\in (-2,2)$ belongs to $A_2$ but we will see in a moment that $\rho$ as in~\eqref{def:weight_two} with $a(r) = r^{\alpha}$ does not.
\end{remark}

\begin{remark}
There are other weights $\rho$ that can be considered for which $\cH \neq \cW$. For example, in~\eqref{def:weight_two} it is possible to replace $a(|x|)$ with a function $\lambda(x)$ bounded away from zero and infinity in $\bbR^2\setminus B(0,\epsilon)$ for all $\epsilon>0$, such that $\lambda,\lambda^{-1} \in L^1_{\loc}(\bbR^2)$ and
 \[
  \int_{B(0,1)} \frac{\lambda(x)}{|x|^2}\,dx < \infty.
 \]
 The proof that for this choice $\cH \neq \cW$ can be done along the same lines of what follows.
\end{remark}

Recall from the introduction that
\[
 \cE(u,v) :=\int_{\bbR^2}\nabla u \cdot \nabla v \;\rho dx,
\]
which is well defined for all $u,v$ belonging to
\begin{equation}\label{def:W}
 \cW := \left \{ u\in W^{1,1}_{\loc}(\bbR^2)\; : \; \cE_1(u,u) <\infty \right \}.
\end{equation}
Also recall that $\cE_1 = \cE + \langle \cdot, \cdot\rangle$ where $\langle \cdot, \cdot\rangle$ denotes the scalar product in $L^2(\bbR^d,\rho dx)$.
Since $\rho,\rho^{-1}\in L^1_{\loc}(\bbR^2)$,  $(\cW,\cE_1)$ is an Hilbert space and  $C_0^\infty(\bbR^2)\subset \cW$.
We denote by $\cH$ the closure of $C_0^\infty(\bbR^2)$ in $\cW$ and by $\cH_0$ the closure of $C_0^\infty(\bbR^2\setminus \{0\})$ in $\cW$.
Clearly,
\begin{equation}\label{eq:inclusions}
 \cH_0\subseteq \cH\subseteq \cW.
\end{equation}

\subsection{Description of the spaces $\cH_0$, $\cH$ and $\cW$}

In the interest of characterizing the inclusions~\eqref{eq:inclusions}, we will proceed  by defining functionals $\ell^+ : \cW\to \bbR$ and  $\ell^- : \cW\to \bbR$ which heuristically are mappings that take functions in $\cW$ to their trace at the origin ``looking'' from $Q_1$ and $Q_3$ respectively.
In what comes next we follow closely~\cite[Section 9.3]{zhikov2011variational} which considered the case $a(r) = r^\alpha$.

We start by properly defining $\ell^+$ and $\ell^-$.
Consider a function $u\in \cW$ and write $u(r,\theta)$ for its representation in polar coordinates. Since $u\in W^{1,1}_{\loc}(\bbR^2)$, then $u$ is also weakly
differentiable with respect to the polar coordinates, in particular we can
choose a version of $u$ which is absolutely continuous in $r$ for almost all
$\theta$. From now on fix one such version. We set
\begin{equation}\label{eq:ubar}
 \bar{u}(r) \ldef \frac{2}{\pi} \int_0^{\pi/2} u(r,\theta) d\theta,
\end{equation}
We remark that $\bar{u}(r)$ does not depend on the version of $u$ which we have just chosen.
By Fubini's theorem and the absolute continuity of $u$ in $r$ we have
\[
 \bar{u}(r_2) - \bar{u}(r_1)  = \frac{2}{\pi} \int_{r_1}^{r_2} \int_0^{\pi/2} \partial_r u(r,\theta) \,d\theta dr.
\]
A simple calculation using Cauchy-Schwartz inequality and that $|\partial_r u|\leq |\nabla u|$ shows that for all $r_1,r_2 \in (0,1]$
\begin{equation}\label{eq:ell_cont}
 | \bar{u}(r_2) - \bar{u}(r_1) |^2 \leq  \frac{2}{\pi} \left(  \int_{r_1}^{r_2} \frac{a(r)}{r} dr \right)\left(  \int_{Q_1} |\nabla u|^2 \rho \,dx \right).
\end{equation}
Assumption~\eqref{ass:integrability} together with~\eqref{eq:ell_cont} implies that $\bar{u}$ is continuous on $(0,1)$ and also that the limit $\lim_{r\to 0} \bar{u}(r)$ exists. We define $\ell^+(u)\ldef \lim_{r\to 0} \bar{u}(r)$ and we note that for all $s\in (0,1]$
\begin{equation}\label{eq:ellplus}
 |\bar{u}(s) - \ell^+(u)|^2 \leq \frac{2}{\pi} \left(  \int_{0}^{s} \frac{a(r)}{r} dr \right)\left(  \int_{Q_1} |\nabla u|^2 \rho \,dx \right).
\end{equation}
From~\eqref{eq:ellplus} we deduce that the map $\ell^+ : \cW\to \bbR$ is linear and continuous. With the same spirit, we define $\ell^{-}: \cW\to \bbR$ by
\[
 \ell^-(u)\ldef \lim_{r\to 0}  \int_0^{\pi/2} u(r,\theta+\pi) d\theta.
\]
Clearly, $\ell^{-}$ has the same properties of $\ell^+$ and satisfies an inequality similar to~\eqref{eq:ellplus} with $Q_3$ in place of $Q_1$.
\begin{lemma}\label{lem:ellonH}
 The maps $\ell^{+},\ell^{-}:\cW \to \bbR$ are linear and continuous. Moreover, for all $u\in \cH$ we have $\ell^{+}(u) = \ell^{-}(u)$.
\end{lemma}
\begin{proof}
 The continuity and linearity was proven above. For the second part, it suffices to notice that  $\ell^{+}(u) = \ell^{-}(u)$ for all $u\in C_0^\infty(\bbR^2)$ which extends to $u\in\cH$ by continuity.
\end{proof}

By making use of the lemma above, we are now able to prove the following Proposition.

\begin{prop}\label{prop:nonreg} $\rho:\Omega \to [0,\infty]$ defined in~\eqref{def:weight_two} and satisfying~\eqref{ass:onepointdeg},~\eqref{ass:integrability}  is not regular, that is $\cW \neq \cH$.
\end{prop}
\begin{proof} It suffices to provide a function in $\cW\setminus \cH$. The
 following example is well known in the literature (see~\cite[section 5]{zhikov1998weighted}), we consider
 \begin{equation}\label{def:psi}
  \psi(x)\ldef\left \{
  \begin{array}{lll}
   1\;,        & x\in Q_1, \\
   x_2/|x| \;, & x\in Q_2, \\
   0\;,        & x\in Q_3, \\
   x_1/|x| \;, & x\in Q_4,
  \end{array}
  \right.
 \end{equation}
 and set
 \[
  \psi_0(x) = {(1-|x|^2)}_+ \psi(x).
 \]
 It can be easily checked that $\psi_0\in \cW$ thanks to~\eqref{ass:integrability}. It is clear that $\ell^{+}(\psi_0) = 1$ and
 $\ell^{-}(\psi_0)=0$. In view of Lemma~\ref{lem:ellonH} it follows that $\psi_0\notin \cH$.
\end{proof}

The careful reader may have observed that in Proposition~\ref{prop:nonreg} we did not assume~\eqref{ass:onerank}. This will be needed below to prove that $\cW$ is obtained from $\cH$ by adding the sole function $\psi_0$ defined in~\eqref{def:psi}.

\begin{prop}\label{prop:ells} Assume~\eqref{ass:onepointdeg},~\eqref{ass:integrability} and~\eqref{ass:onerank}. An element $u\in \cW$ belongs to $\cH$ if and only if $\ell^+(u) = \ell^-(u)$.
 Moreover, an element $u\in \cW$ belongs to $\cH_0$ if and only if $\ell^+(u) = \ell^-(u) = 0$.
\end{prop}
\begin{proof} Let $u\in \cW$ be such that  $\ell^+(u) = \ell^-(u)$, we want to prove that $u\in\cH$. In view of~\eqref{ass:onepointdeg}, we can assume that $u\in \cW$ is compactly supported in $B(0,1)$.
 We can also assume that $u$ is bounded and $\ell^+(u) = 0$ by considering
 \[
  v_k(x) \ldef u_k(x) - \psi_0(x) \Big(\ell^+(u_k) - \ell^-(u_k)\Big) -\eta(x) \ell^{-}(u_k)
 \]
 in place of $u$ where $\eta\in C_0^\infty(B(0,1))$, $\eta(0)=1$ and $u_k=-k\vee(u\wedge k)$. Indeed, if $v_k\in \cH$ then $v_k$ converges  to  $u-\eta(x) \ell^{-}(u)\in\cH$ with respect to $\cE_1$ and thus also $u\in \cH$.

 We now show that $u$ can be approximated with functions in $\cH_0$  with respect to $\cE_1$. Consider for all $\epsilon>0$ a smooth function $\phi_\epsilon$ with values in $[0,1]$ such that $\phi\equiv 0$ for $|x|\leq \epsilon$, $\phi\equiv 1$ for $| x|>2\epsilon$ and  $|\nabla \phi_\epsilon| \leq 2/\epsilon$. Set $u_\epsilon \ldef \phi_\epsilon u $.
 Since $\supp [u_\epsilon] \subset \bbR^2\setminus B(0,\epsilon)$, the density of smooth functions in the
 classical Sobolev space implies that $u_\epsilon \in \cH_0$.
 By the dominated convergence theorem we readily obtain that $ u_\epsilon \to u$
 in $L^2(\bbR^2,\mu)$. To conclude that $u$ is in $\cH_0$ it suffices to check the weak convergence of
 gradients
 \[
  \nabla u_\epsilon =   u\nabla \phi_\epsilon + \phi_\epsilon \nabla u \rightharpoonup \nabla u,\quad \mbox{ in } L^2(\bbR^2,\mu),
 \]
 which can be obtained, up to extracting a subsequence, if we show that $\nabla u_\epsilon$ are uniformly bounded in $L^2(\bbR^2,\mu)$.

 Since $\phi_\epsilon \nabla u$ is trivially bounded in $L^2(\bbR^2,\mu)$
 uniformly in $\epsilon$, we can focus on $u \nabla \phi_\epsilon$. Moreover, we
 can restrict our analysis on $Q_1$ and $Q_2$, being the argument for $Q_3$ and
 $Q_4$ similar.

 Let us start with $Q_1$. We consider $\bar{u}(r)$ as defined in~\eqref{eq:ubar} and set $v(x) \ldef u(x)-\bar{u}(r)$. Then, for a constant $C>0$ that does not depend on $\epsilon$ and may change from line to line, we have
 \begin{align*}
  \int_{Q_1} |u \nabla\phi_\epsilon |^2 & a^{-1}(|x|)\,dx   \leq \frac{C}{\epsilon^2}  \int_{\epsilon}^{2\epsilon} \int_0^{\pi/2}
  |u(r,\theta)|^2 a^{-1}(r) r\,d\theta dr                                                                                                                                                                                                  \\
                                        & \leq \frac{C}{\epsilon^2}  \int_{\epsilon}^{2\epsilon} |\bar{u}(r)|^2 a^{-1}(r)r\,dr +  \frac{C}{\epsilon^2}\int_{\epsilon}^{2\epsilon} \int_{0}^{\pi/2} |v(r,\theta)|^2 a^{-1}(r)r\,d\theta dr.
 \end{align*}
 We observe that $\frac{2}{\pi} \int_{0}^{\pi/2}v(r,\theta) \,d\theta = 0$. Hence, an application of Poincar\'e's inequality yields
 \begin{equation}\label{eq:poincare}
  \int_{0}^{\pi/2} |v(r,\theta)|^2\,d\theta \leq  C \int_{0}^{\pi/2} |\partial_\theta v(r,\theta)|^2 \,d\theta \leq C \int_{0}^{\pi/2} r^2 |\nabla u|^2\,d\theta,
 \end{equation}
 in view of $|\nabla u|^2 = |\partial_r u|^2 + |\partial_\theta u|^2/r^2$ and $\partial_\theta u = \partial_\theta v$. Integrating~\eqref{eq:poincare} against $a^{-1}(r)r$ on $(\epsilon,2\epsilon)$, and using that $r^2\leq 4\epsilon^2$ allows for the bound
 \[
  \frac{1}{\epsilon^2}\int_{\epsilon}^{2\epsilon} \int_{0}^{\pi/2} |v(r,\theta)|^2 a^{-1}(r)r\,d\theta dr \leq C \int_{Q_1} |\nabla u|^2\rho \,dx.
 \]
 Furthermore, an application of`\eqref{eq:ellplus} gives
 \[
  \frac{1}{\epsilon^2}  \int_{\epsilon}^{2\epsilon} |\bar{u}(r)|^2 a^{-1}(r)r\,dr \leq \frac{C}{\epsilon^2}\left( \int_{\epsilon}^{2\epsilon}a^{-1}(r)r
  \int_{0}^{r} \frac{a(y)}{y}\, dy dr\right) \left( \int_{Q_1} |\nabla u|^2\rho \,dx \right),
 \]
 thus the uniform bound follows from~\eqref{ass:onerank}.

 Let us now proceed with $Q_2$. Exploiting the boundness of $u$ we get
 \begin{align*}
  \int_{Q_2} |u \nabla\phi_\epsilon |^2 & a(|x|)\,dx   \leq \frac{C}{\epsilon^2} \int^{2\epsilon}_{\epsilon}  a(r)r\,dr \leq  C \int^{2\epsilon}_{\epsilon}  \frac{a(r)}{r}\,dr
 \end{align*}
 which is uniformly bounded in $\epsilon\in (0,1)$ by~\eqref{ass:integrability}.
\end{proof}

\begin{corro}\label{corro:codimension} $\cH_0$ has codimension one in $\cH$ and $\cH$ has codimension one in $\cW$. More precisely, $\cW = \cH+\bbR \psi_0$.
\end{corro}
\begin{proof}
 By Proposition~\ref{prop:ells} we see immediately that, given $u\in \cW$ and setting $\lambda := \ell^+(u)-\ell^-(u)$, the function $v = u - \lambda\psi_0$
 is in $\cH$ since $\ell^{+}(v) = \ell^{-}(v)$ and therefore $\cH$ has codimension one in $\cW$.
 Similarly, fix any $\eta \in C_0^\infty(\bbR^2)$ with $\eta(0) = 1$, then for any function $u\in \cH$, $u-\ell^+(u)\eta$, belongs to $\cH_0$ by Proposition~\ref{prop:ells}.
\end{proof}

From now on we will always assume that~\eqref{ass:onepointdeg},~\eqref{ass:integrability} and~\eqref{ass:onerank} are satisfied. It turns out that a better understanding of the functions contained in $\cH_0$ will be useful to study the Hunt process associated with the Dirichlet form $(\cE,\cH_0)$ on $L^2(\bbR^2, \mu)$.
In the next proposition we explicitly identify a class of functions that belong to $\cH_0$.

\begin{prop}\label{prop:domain} Set $S_1:= \bbR/2\pi\bbZ$, let $\eta \in C_0^\infty(\bbR^2)$ and  $f: S_1 \to \bbR$ be piecewise differentiable and such that $f\equiv 0$ on $(-\pi,-\pi/2)\cup(0,\pi/2)$.
 Then, the function defined by $u(0)\ldef 0$ and $u(x)\ldef \eta(x)f(\theta(x))$ for  $x\in\bbR^2\setminus \{ 0\}$ belongs to $\cH_0$.
\end{prop}
\begin{proof} In view of Proposition~\eqref{prop:ells} it suffices to show that $u\in \cW$, since it is clear from the assumptions that $\ell^+(u) = \ell^-(u) = 0$.
 Since $u$ is bounded and compactly supported, $\|u\|_{2} <\infty$. Moreover, noticing that $u$ is zero on $Q_1\cup Q_3$, we have
 \begin{equation}\label{eq:L2bound}
  \int_{\bbR^2}|\nabla u |^2\,\rho dx \leq \int_{Q_2\cup Q_4} |f\nabla \eta|^2 a(|x|)\,dx + \int_{Q_2\cup Q_4} |\eta\nabla f|^2 a(|x|)\,dx,
 \end{equation}
 which is bounded by recalling that $|\nabla f|^2 = |\partial_{\theta} f|/r^2$ and
 thanks to~\eqref{ass:integrability}.
\end{proof}

In the sequel we want to study how the stochastic process associated to $(\cE,\cH)$ approaches the origin, for that we estimate the capacity of cones in $Q_2$ and $Q_4$ (see figure below). For $\epsilon,\delta>0$ we define using polar coordinates $x = (r,\theta)$ the two cones
\begin{equation}\label{eq:cones}
 A_{\epsilon,\delta}^+\ldef (0,\epsilon)\times(\pi/2+\delta,\pi-\delta),\quad A_{\epsilon,\delta}^-\ldef(0,\epsilon)\times(-\pi/2+\delta,-\delta),
\end{equation}
so that $A_{\epsilon,\delta}^-$ is obtained by $A_{\epsilon,\delta}^+$ with a rotation of angle $\pi$ around the origin.

\begin{figure}[ht]
 \centering
 \includegraphics[width=0.5\textwidth]{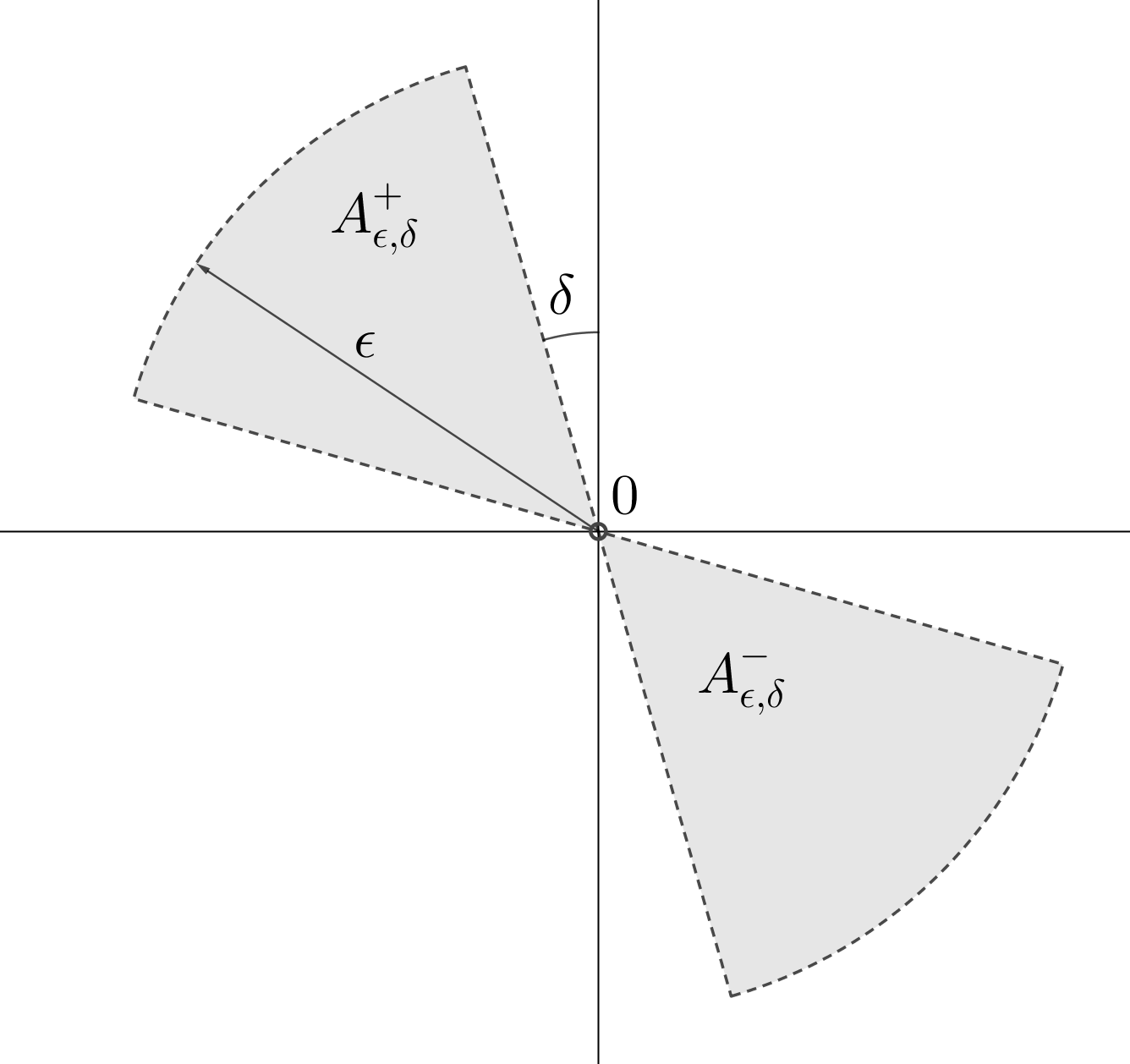}
\end{figure}

\begin{lemma} For all $\epsilon,\delta \in (0,1)$ it holds that
 \begin{equation}\label{ineq:capacityofcones}
  \capa_1(A_{\epsilon,\delta}^+)\leq \frac{C}{\delta} \int_{0}^{2\epsilon} \frac{a(r)}{r} \,dr,\quad
  \capa_1(A_{\epsilon,\delta}^-)\leq \frac{C}{\delta} \int_{0}^{2\epsilon} \frac{a(r)}{r} \,dr,
 \end{equation}
 for some constant $C$ independent of $\epsilon$ and $\delta$. Here the capacity is the one associated to $(\cE,\cH)$ on $L^2(\Omega,\mu)$.
\end{lemma}
\begin{proof}We prove the inequality for $A_{\epsilon,\delta}^+$, being the argument for $A_{\epsilon,\delta}^-$ completely analogous.
 To find an upper bound on the capacity of the open set $A_{\epsilon,\delta}^+$, it suffices to compute the Dirichlet energy of a function which is in $\cH_0$ and which is equal to $1$ on $A_{\epsilon,\delta}^+$.
 Let $u = u(\theta,r)$ be equal to $1$ on $A_{\epsilon,\delta}^+$, such that $|\partial_\theta u|\leq C/\delta$, $|\partial_r u|\leq C/\epsilon$ for some constant $C>0$ and equal to zero on $\bbR^2\setminus A^+_{2\epsilon,0}$.
 Also we can clearly assume that $u$ takes values in $[0,1]$.
 Then, since $u$ is bounded by $1$ and supported in $A^+_{2\epsilon,0}$,
 \[
  \int_{\bbR^d} |u|^2 d\mu \leq \int_{A^+_{2\epsilon,0}} \rho \,dx =\frac{2}{\pi}\int_{0}^{2\epsilon} \frac{a(r)}{r} \,dr.
 \]
 We now observe that
 \begin{align*}
  \int_{\bbR^d} |\nabla u|^2 d\mu & = \int_{\pi/2}^{\pi} \int_0^{2\epsilon} \left(|\partial_r u|^2 + \frac{|\partial_\theta u|^2}{r^2}\right) a(r) r \,dr d\theta                                                             \\
                                  & = \int_{\pi/2}^{\pi} \int_0^{2\epsilon} |\partial_r u|^2 a(r) r \,dr d\theta + 2\int_0^{2\epsilon} \int_{\pi/2}^{\pi/2+\delta} \frac{\sup|\partial_\theta u|^2}{r^2} a(r) r \,dr d\theta.
 \end{align*}
 By using $ |\partial_r u| \leq C/\epsilon$ we can bound the first integral by
 \[
  \int_{\pi/2}^{\pi} \int_0^{2\epsilon} |\partial_r u|^2 a(r) r \,dr d\theta \leq  C^2 \frac{\pi}{2}  \int_0^{2\epsilon} \frac{a(r)}{\epsilon^2} r \,dr \leq 2 C^2\pi \int_0^{2\epsilon} \frac{a(r)}{r} \,dr
 \]
 where in the second inequality we used the fact that $1/\epsilon^2\leq 4/r^2$ in the domain of integration. We finally use $ |\partial_\theta u| \leq C/\delta$  to bound the second integral by
 \[
  \int_0^{2\epsilon} \int_{\pi/2}^{\pi/2+\delta} \frac{|\partial_\theta u|^2}{r^2} a(r) r \,dr d\theta \leq \int_0^{2\epsilon} \int_{0}^{\delta} \frac{C^2}{\delta^2} \frac{a(r)}{r} \,dr d\theta
  =  \frac{C^2}{\delta} \int_0^{2\epsilon}\frac{a(r)}{r} \,dr,
 \]
 and the conclusion of the lemma follows easily by putting all the estimates together.
\end{proof}

\subsection{The $\cH_0$- and $\cH$-processes}\label{sec:H0H}
Now that we have a better understanding of $\cH_0$, $\cH$ and $\cW$ under~\eqref{ass:onepointdeg},~\eqref{ass:integrability} and~\eqref{ass:onerank}, we are ready to discuss about the processes associated to them. We start with the Dirichlet form $(\cE, \cH)$ on $L^2(\bbR^2,\mu)$.

\begin{prop}\label{prop:properties} $(\cE, \cH)$ on $L^2(\bbR^2,\mu)$ is a strongly local regular Dirichlet form. Moreover, $(\cE,\cH)$ is recurrent and conservative.
\end{prop}
\begin{proof} The proof of the fact that $(\cE, \cH)$ is a strongly local Dirichlet form is standard and we omit it here, we only mention that $\rho,\rho^{-1}\in L^1_{\loc}(\bbR^2)$ is required. Regularity is obvious as  by definition $\cH$ is the closure of $C_0^\infty(\bbR^2)$ in $(\cW,\cE_1)$.

 According to~\cite[Theorem 1.6.3]{fukushima2011dirichlet}, recurrence is equivalent to showing that there exists $u_n\in \cH$ such that $u_n\to 1$ almost surely and $\cE(u_n,u_n)\to 0$ as $n\to \infty$.
 It is immediate to verify that the sequence of functions
 \begin{equation*}
  u_n(x)\ldef\left \{
  \begin{array}{lll}
   1\;,                              & |x| \leq n,    \\
   \frac{2\log n-\log|x|}{\log n}\;, & n<|x|\leq n^2, \\
   0\;,                              & |x|>n^2
  \end{array}
  \right.
 \end{equation*}
 has the required property. This is not surprising, because by construction the process is a Brownian motion in $\{|x|>1\}$, which is recurrent in dimension $d=2$.
 Clearly, conservativeness follows directly from recurrence~\cite[Lemma 1.6.5]{fukushima2011dirichlet}.
\end{proof}

By Theorem~\ref{thm:huntprocess} there exists a Hunt process ${(\{X_t \}}_{t\geq0},P_x)$, $x\in \bbR^2$ which is uniquely determined up to a properly exceptional set and which is associated to $(\cE, \cH)$. Moreover, by means of the strong locality, this process can be taken to have continuous sample paths $P_x$-almost surely for all $x\in \bbR^2$.

We come now to the study of the process $(X,P_x)$ near the origin. We set
\begin{equation}\label{eq:hittingtime}
 \sigma \ldef \inf \{ t>0\; : \; X_t = 0\}
\end{equation}
to be the hitting time of $0$ of the process $X$. The hitting probability and the $\alpha$-order hitting probability are denoted by $\varphi(x)$ and $u_\alpha(x)$, $\alpha>0$ respectively
\begin{equation}\label{eq:hitprob}
 \varphi(x) \ldef P_x(\sigma < \infty),\qquad u_\alpha (x) \ldef E_x[e^{-\alpha \sigma}],\quad x\in \bbR^2.
\end{equation}
We start with the following simple but fundamental lemma.
\begin{lemma}
 The origin has positive $\cH$-capacity.
\end{lemma}
\begin{proof} Suppose by absurd that $\capa_1(\{0\}) =0$. By assumption~\eqref{ass:onepointdeg}, $\rho$ is degenerate only at the origin. The Lebesgue measure of $\{ 0\}$ is trivially zero, thus it follows from Proposition~\ref{prop:onepoint} that $\rho$ is regular, which is a contradiction to the conclusion of Proposition~\ref{prop:nonreg}.
\end{proof}

From the fact that $\capa_1(\{ 0\})>0$ we can now derive the same conclusions as in~\cite[Section 2]{fukushima2005poisson}. It implies that $u_\alpha$ is a non-trivial element of $\cH$ and the $\alpha$-potential $U_\alpha \nu_\alpha$ of a positive measure $\nu_\alpha$ concentrated on $\{ 0\}$ (see~\cite[Section 2.2]{fukushima2011dirichlet}),
\begin{equation*}
 \cE_\alpha(u_\alpha, v) = \tilde{v} (0)\nu_\alpha(\{ 0\}),\quad v\in \cH,
\end{equation*}
where we recall that $\tilde{v}$ denotes a quasi-continuous version of $v$. In particular
\begin{equation}
 \cE_1(u_1,u_1) = \capa_1(\{0\}) >0.
\end{equation}

It is not hard to show (\cite[Lemma 2.3.4]{fukushima2011dirichlet}) that $\cH_0 = \{ u \in \cH : \tilde{u}(0) = 0 \}$, and that $(\cE,\cH_0)$ is a regular strongly local Dirichlet form on $L^2(\bbR^2_{0} ,\mu)$, where we noted $\bbR^2_{0} \ldef \bbR^2\setminus \{ 0\}$ to shorten notation.
$(\cE,\cH_0)$ is associated with the part $(X^0,P_x)$ of $X$ on the set $\bbR^2_{0}$. That is, the diffusion process $X^0$ is obtained from $X$ by killing upon the hitting time $\sigma$ (see Section~\ref{sec:elements}). Since $(X,P_x)$ is conservative it is clear that the lifetime of $X^0$ coincides with $\sigma$ and that we can rewrite~\eqref{eq:hitprob} in terms of the killed process. We denote by
\begin{equation*}
 p_t f(x) \ldef E_x[f(X_t)],\quad G_\alpha f(x) \ldef E_x\left [\int_0^\infty e^{-\alpha t} f(X_t)\,dt\right], \quad x\in \bbR^2,
\end{equation*}
the transition function and the resolvent of $X$. Similarly, we note by $p^0_t$ and $G^0_\alpha$ the same quantities for $X^0$.
\begin{prop}\label{prop:irreducible}
 The Dirichlet forms $(\cE,\cH_0)$ and $(\cE,\cH)$ are irreducible. In particular, $\varphi(x) = 1$ and $u_\alpha(x)>0$ for q.e.\ $x\in \bbR^2$.
\end{prop}
\begin{proof} The fact that the Dirichlet form $(\cE,\cH_0)$ on $L^2(\bbR^2_{0},\mu)$ is irreducible follows immediately from assumption~\eqref{ass:onepointdeg} and~\cite[Example 4.6.1]{fukushima2011dirichlet}.
 To prove that $(\cE, \cH)$ is irreducible it is enough to show that for any two Borel sets $E, F\subset \bbR^2$ of positive measure $\langle 1_E, G_\alpha 1_F \rangle > 0$. Clearly,
 \begin{equation*}
  \langle 1_E, G_\alpha 1_F \rangle \geq   \langle 1_E, G^0_\alpha 1_F \rangle >0
 \end{equation*}
 by the irreducibility of $(\cE,\cH_0)$.

 Finally, as $(\cE,\cH)$ is irreducible and recurrent, it follows from~\cite[Theorem 4.7.1]{fukushima2011dirichlet} that $\varphi(x) = 1$ for q.e.\ $x\in \bbR^2$, and thus that $u_\alpha(x)>0$ for q.e.\ $x\in \bbR^2$.
\end{proof}

\begin{remark}\label{rem:onepoint}We shall remark here that $(X,P_x)$ on $\bbR^2$ is nothing else but a symmetric one point extension of $(X^0,P_x)$ on $\bbR^2_0$. That is,
 \begin{itemize}
  \item $X$ is a $\mu$-symmetric diffusion process on $\bbR^2$ with no killing inside $\bbR^2$;
  \item $X$ is an extension of $X^0$ in the sense that the process obtained from $X$ by killing upon the hitting time of zero is identical in law to $X^0$.
 \end{itemize}
 One-point symmetric extensions have been extensively studied in a series of papers~\cite{fukushima2005poisson, chen2005extending, chen2015one} at different levels of abstraction (see also the monograph~\cite{chen2012symmetric}). One important lesson we learn from these papers is that such extensions are unique in law~\cite[Theorem 7.5.4]{chen2012symmetric}. This implies in particular that $(\cE,\cW)$ on $L^2(\bbR^2,\mu)$ cannot be quasi-regular.
\end{remark}

In the next theorem we characterize the resolvent of $X$ via quantities which depend solely on the killed process $X_0$.
\begin{theorem}[Theorem 2.1~\cite{fukushima2005poisson}]\label{thm:char} It holds that
 \begin{itemize}
  \item[i)] $u_\alpha$ is a non-trivial element in $\cH\cap L^1(\bbR^2_{0},\mu)$.
  \item[ii)] For any $f\in L^2(\bbR^2,\mu)$ and $x\in \bbR^2$,
        \begin{equation}\label{eq:generatordecomposition}
         G_\alpha f(x) = G_\alpha^0 f(x) + \frac{\langle u_\alpha, f\rangle}{\alpha\langle u_\alpha, 1\rangle} u_\alpha(x),\quad G_\alpha f(0) = \frac{\langle u_\alpha, f\rangle}{\alpha\langle u_\alpha, 1\rangle}.
        \end{equation}
  \item[iii)] The origin is regular for itself and an instantaneous state with respect to $X$
        \begin{equation*}
         P_0(\sigma = 0, \tau = 0) = 1, \quad \tau \ldef\inf \{t> 0: X_t\neq 0\}.
        \end{equation*}
 \end{itemize}
\end{theorem}

We proved in Proposition~\ref{prop:irreducible} that the origin is a recurrent point for $X$. We will now use that the maps $t\to \tilde{u}(X_t)$ are continuous whenever $u\in \cH$ to give a qualitative description on how the paths approach the origin. First we start with a Lemma.

\begin{lemma}\label{lem:qc}
 Let $u:\bbR^2\to\bbR$ be defined as in Proposition~\ref{prop:domain}.
 Then, $u$ is quasi-continuous in the restricted sense.
\end{lemma}
\begin{proof}
 We know from Proposition~\ref{prop:domain} that $u\in \cH_0$. Moreover, by construction, $u\in C(\bbR^2_{0})$ and $u(0)=0$. According to~\cite[Theorem 2.1.3]{fukushima2011dirichlet} there is a $\mu$-modification $\tilde{u}$ of $u$ which is quasi-continuous in the restricted sense. Using continuity of $u$ and the definition of quasi-continuity for $\tilde{u}$, it is immediate to check that $u$ is itself quasi-continuous in the restricted sense.
\end{proof}

\begin{prop}\label{prop:approach_origin} There exists a properly exceptional set $\cN\subset \bbR^2$ such that for all $x\in\bbR^2_{0}\setminus\cN$
 \begin{equation} \label{eq:entrance}
  \Big[\liminf_{t\uparrow\sigma} \theta(X^0_t),\limsup_{t\uparrow\sigma} \theta(X^0_t)\Big]\subset [-\pi,-\pi/2]\cup[0,\pi/2],\quad P_x\mbox{-a.s}.
 \end{equation}
 where $\theta : \bbR^d\setminus\{0\} \to [-\pi,\pi)$ is the angle variable in the polar coordinates.
\end{prop}
\begin{remark} Proposition~\ref{prop:approach_origin} shows that $P_x\mbox{-a.s}$ the angular component of  the process $X_t$ associated to $(\cE, \cH)$ remains in a arbitrarily small neighborhood of either $[-\pi,-\pi/2]$ or $[0,\pi/2]$ for times immediately before $\sigma$.
In particular the origin is approached only from the cones $Q_1$ or $Q_3$.
\end{remark}
\begin{proof} Let $\eta\in C_0^\infty(\bbR^2)$ be such that $\eta(x)= 1$ for $x\in B(0,1)$. We define the function,
\begin{equation}\label{eq:test}
 \xi(x) \ldef \eta(x) f(\theta(x)),\quad x\neq 0,\qquad \xi(0) \ldef 0,
\end{equation}
where $f:S_1\to \bbR$ is the piecewise differentiable function defined by
\begin{equation*}
 f(\theta)\ldef\left \{
 \begin{array}{lll}
  0\;,                      & \theta \in [-\pi,-\pi/2), \\
  -\pi/4+|\theta-\pi/4|\;,  & \theta \in [-\pi/2, 0),   \\
  0\;,                      & \theta \in [0,\pi/2),     \\
  \pi/4-|\theta-3/4 \pi|\;, & \theta \in [\pi/2,\pi).
 \end{array}
 \right.
\end{equation*}
By Lemma~\ref{lem:qc}, $\xi$ is quasi-continuous in the restricted sense. By Proposition~\ref{continuity} and since $X^0$ has continuous paths, there exists a properly exceptional set such that for all $x\in \bbR_0^2\setminus\cN$
 \begin{equation*}
  \lim_{s\to t} \xi(X^0_s) = \xi(X^0_{t}),\, \forall t\in (0,\infty),\quad P_x\mbox{-a.s.}
 \end{equation*}
 In particular for all $x\in \bbR_0^2\setminus\cN$, as $\sigma<\infty$, $P_x$-a.s.
 \begin{equation*}
  \lim_{t\uparrow \sigma} \xi(X^0_t) = \xi(X_\sigma) = \xi(0)= 0,\quad P_x\mbox{-a.s.}
 \end{equation*}
 Fix any sample path in such a set of $P_x$-measure one. First notice that $X^0_t\to 0$ as $t\uparrow \sigma$.
 We prove that $\limsup_{t\uparrow \sigma} \theta(X_t^0) \in [-\pi,-\pi/2]\cup [0,\pi/2]$. If this were not the case there would exist a sequence $t_n \uparrow \sigma$ such that
 \[
  \theta(X^0_{t_n})\to \bar{\theta} \notin [-\pi,-\pi/2]\cup [0,\pi/2].
 \]
 It yields the following contradiction
 \begin{equation*}
  \lim_{t\uparrow \sigma} \xi(X^0_t) =\lim_{n\to\infty} \eta(X_{t_n}^0) f(\theta(X_{t_n}^0)) = f(\bar{\theta})\neq 0.
 \end{equation*}
 Clearly, the same argument works for the inferior limit. For the second part, suppose that
 \begin{equation*}
  \liminf_{t\uparrow \sigma} \theta(X_t^0) \in [-\pi,\pi/2],\quad \limsup_{t\uparrow \sigma} \theta(X_t^0) \in [0,\pi/2].
 \end{equation*}
 As $t\to \theta(X^0_{t})$ is continuous for $t<\sigma$, it follows again that there exists a sequence $t_n \uparrow \sigma$ such that
 \[
  \theta(X^0_{t_n})\to \bar{\theta} \notin [-\pi,-\pi/2]\cup [0,\pi/2],
 \]
 which again leads to a contradiction.
\end{proof}

\subsection{The $\cW$-process}\label{sec:W}

According to~\cite[Theorem 7.5.4]{chen2012symmetric} there is a unique quasi-regular one-point extension of $\cH_0$. We conclude that $(\cE,\cW)$ is neither a regular or quasi-regular Dirichlet form on $L^2(\bbR^2,\mu)$, because in that case it would coincide with $(\cE,\cH)$.
Nonetheless, we wish to associate a concrete stochastic process to it. This is achieved by a regularization procedure; roughly speaking, we construct a regular Dirichlet form on a possibly different space which is ``isomorphic'' to the original one. We start by recalling briefly the notion of equivalent Dirichlet spaces; for more see~\cite[Appendix A.4]{fukushima2011dirichlet}.

For the purposes of this article, we say that $(E,\mu,\cE,\cF)$ is a \emph{Dirichlet space} if $E$ is a locally compact metric space, $\mu$  a positive Radon measure on $E$ such that $\supp[\mu] = E$ and $(\cE,\cF)$ is a Dirichlet form on $L^2(E,\mu)$.
We shall denote by $\cF_b \ldef \cF \cap L^{\infty}(E,\mu)$, and by $\|\cdot \|_{\infty}$ the $\mu$-essential supremum.

We call two Dirichlet spaces $(E,\mu,\cE,\cF)$ and $(F,\nu,\cD,\cG)$ equivalent if there is an algebraic isomorphism $\Phi:\cF_b\to\cG_b$ which preserves the following metrics, for $u\in \cF_b$
\[
 \|u\|_{\infty} = \|\Phi(u)\|_{\infty},\quad\langle u,u\rangle_E = \langle\Phi(u),\Phi(u)\rangle_F,\quad \cE(u,u) = \cD(\Phi(u),\Phi(u)).
\]

From general theory, for any given Dirichlet space there exists one which is regular and equivalent to it.
Our objective with the next few lemmas is to provide a concrete regular representation of $(\cE, \cW)$ on $L^2(\bbR^2,\mu)$. In doing so, we will first describe the new state space and later the map $\Phi$.

The main idea can be summarized as follows.  We know from Corollary~\ref{corro:codimension} that $\cW$  can be obtained from $\cH$ by adding $\psi_0$. Moreover, $C_0(\bbR^2)$ is dense in $\cH$ and  $\psi_0$ is continuous in $\bbR^2_0$.
Thus, to obtain a set of continuous functions which is dense in $\cW$ with respect to $\cE_1$, it suffices to modify the topology of $\bbR^2$ around the origin in order to make $\psi_0$ continuous.
Practically, this can be achieved by considering $\bbR^2_0$ with the metric $d_\cW(x,y):= |x-y|+|\psi_0(x)-\psi_0(y)|$ and completing it. Below, we perform such completion by hand.

We start from the space $\bbR^2_0$ and we enlarge it by adding the points $0_+$, $0_-$ and the interval $(0,\pi/2)$. More precisely, we define the following sets
\begin{equation*}
 \cO_\Theta \ldef (0,\pi/2),\quad \cO \ldef \{0_+,0_-\}\cup \cO_\Theta,
\end{equation*}
and the new state space $\bbR^2_\cW$ by the disjoint union $\bbR^2_\cW  \ldef \bbR^2_0 \cup \cO$.
We have the following trivial inclusions
\[
 \bbR^2_0\subset \bbR^2,\qquad\bbR^2_0\subset \bbR^2_\cW,
\]
but it is not true that $\bbR^2\subset \bbR^2_\cW$. We define a projection map $\pi :\bbR^2_\cW\to \bbR^2$ by
\begin{equation}\label{projection}
 \pi(x)\ldef x,\,\forall x\in \bbR^2_0,\qquad \pi(x) \ldef 0,\,\forall x\in \cO.
\end{equation}
We extend $\psi_0$ as defined in~\eqref{def:psi} to a map defined on the whole $\bbR^2_\cW$ by  setting
\[
 \psi_0(0_+) \ldef 1,\quad \psi_0(0_-) \ldef 0,\quad \psi_0(z) \ldef \sin(z), \forall z \in \cO_\Theta.
\]
The construction of $\bbR^2_\cW$ and of the extension of $\psi_0$ should be understood in the following way. We notice that $\psi_0(x)= {(1-|x|^2)}_+ h(\theta(x))$ with
\begin{equation*}
 h(\theta)\ldef\left \{
 \begin{array}{lll}
  0\;,          & \theta \in [-\pi,-\pi/2], \\
  \cos(\theta), & \theta \in (-\pi/2, 0),   \\
  1\;,          & \theta \in [0,\pi/2],     \\
  \sin(\theta), & \theta \in (\pi/2,\pi),
 \end{array}
 \right.
\end{equation*}
then we extend $\psi_0(r,\theta)$, thought as a function  $\psi_0:]0,\infty[\times S_1\to\bbR$, to a continuous function in $[0,\infty[\times S_1$. Finally, we identify all the points in $\{ 0 \}\times S_1$ that have the same values under this extension (see Figure~\ref{fig:identification}).
The set of these points and  $\bbR^2_0$ form $\bbR^2_\cW$.
\begin{figure}[ht]
 \centering
 \includegraphics[width=\textwidth]{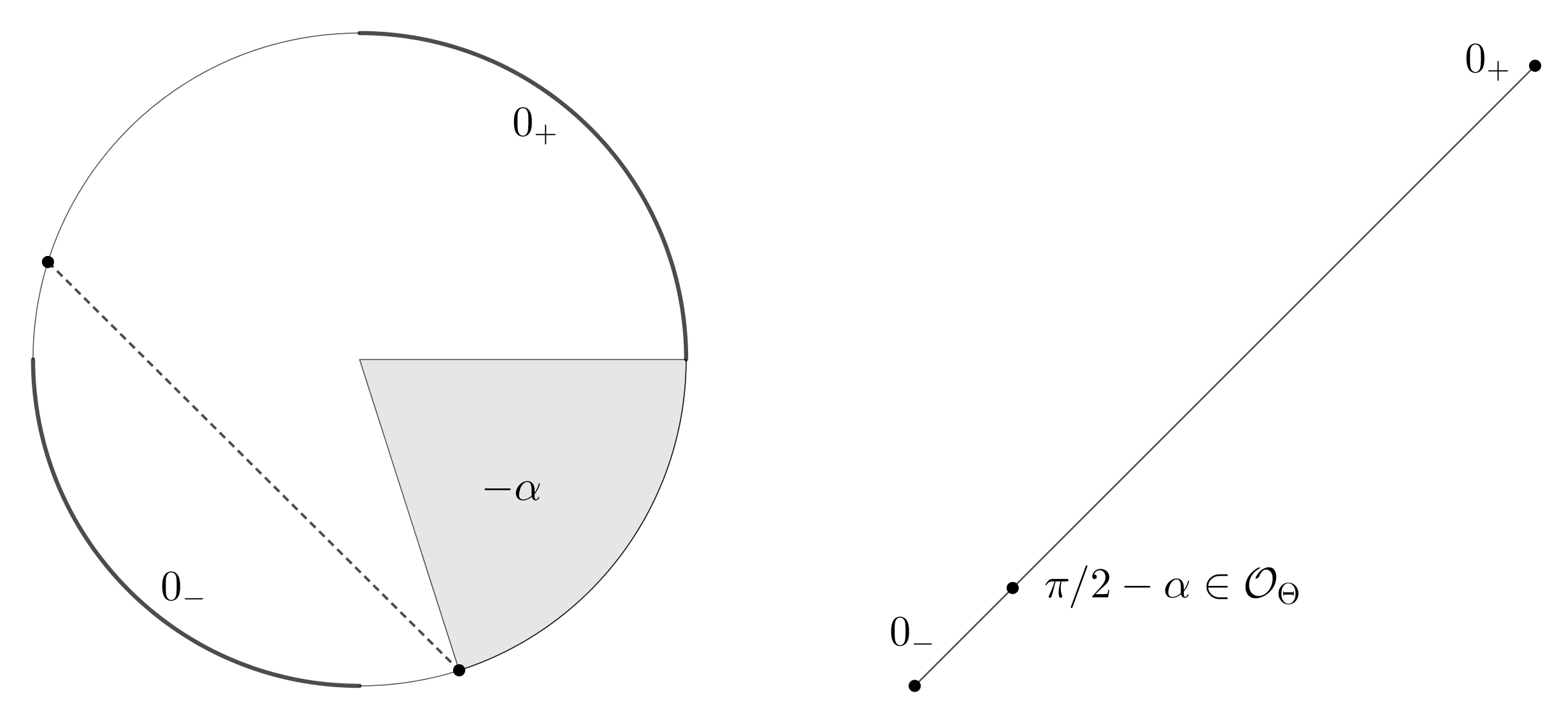}
 \caption{The bold points are those that are going to be identified.}\label{fig:identification}
\end{figure}

We are ready to state the following

\begin{prop} The map $d_{\cW} : \bbR^2_\cW\times \bbR^2_\cW \to [0,\infty)$
 defined by
 \begin{equation}\label{eq:metric}
  d_\cW(x,y) = |\pi(x)-\pi(y)| + |\psi_0(x)-\psi_0(y)|
 \end{equation}
 is a distance on $\bbR^2_{\cW}$. The metric space $(\bbR^2_\cW,d_{\cW})$ is locally compact, separable and complete. Moreover $d_\cW$ and $|\cdot|$ induce the same topology on $\bbR^2_0$.
\end{prop}
\begin{proof} The proof is rather standard and it is left to the reader.
\end{proof}
\begin{remark} $\psi_0$ and $\pi$ are continuous with respect to $(\bbR^2_\cW,d_\cW)$ by construction. Also, for any $u\in C(\bbR^2)$ we have that $u\circ \pi :\bbR^2_\cW\to \bbR$ is continuous with respect to $d_\cW$ being the composition of two continuous functions.
\end{remark}

On $(\bbR^2_\cW,d_\cW)$ we consider the $\sigma$-algebra of Borel $\cB(\bbR^2_\cW)$, which can be easily seen to be the $\sigma$-algebra generated by $\cB(\bbR^2_0)$ the sets $\{ 0_+\}$, $\{ 0_-\}$ and the Borel $\sigma$-algebra $\cB(\cO_\Theta)$.
We can extend the measure $\mu = \rho dx$ from $\cB(\bbR^2_0)$ to $\cB(\bbR^2_\cW)$ by setting $\mu( \cO) = 0$.

The space $L^2(\bbR^2,\mu)$ can be identified with $L^2(\bbR^2_\cW,\mu)$. Indeed, any element $u\in L^2(\bbR^2,\mu)$ is associated uniquely to an element $\Phi(u)\in L^2(\bbR^2_\cW,\mu)$ such that $u=\Phi(u)$ on $\bbR^2_0$.
The map $\Phi :  L^2(\bbR^2,\mu) \to L^2(\bbR^2_\cW,\mu)$ is clearly an isometry. Now we can define $\oW = \Phi(\cW)$ and $\oE:\oW\times\oW\to [0,\infty]$ by
\[
 \oE(\Phi(u),\Phi(u)) \ldef \cE(u,u).
\]
By construction, the Dirichlet spaces $(\bbR^2,\mu,\cE,\cW)$ and $(\bbR^2_\cW,\mu,\oE,\oW)$ are equivalent via the map $\Phi$.
Since the spaces $(\cW,\cE_1)$ and $(\overline{\cW},\overline{\cE}_1)$ are isomorphic, from now on we shall drop the overline and write $(\cW,\cE_1)$ in place of $(\overline{\cW},\overline{\cE}_1)$. On the other hand we shall always stress on which state space these forms are considered, as the topology plays a fundamental role in the description of the Markov process.
In fact, even though $L^2(\bbR^2,\mu)$ can be identified with $L^2(\bbR^2_\cW,\mu)$, it is not true that $C_0(\bbR^2)$ can be identified with $C_0(\bbR^2_\cW)$.
One rather has the continuous inclusion $C_0(\bbR^2)\hookrightarrow C_0(\bbR^2_\cW)$ via the map $\Pi$ that takes $u:\bbR^2\to\bbR$ to $u\circ \pi  \rdef \Pi(u):\bbR^2_\cW\to \bbR$.

\begin{prop}\label{prop:regularity}
 The symmetric form $(\cE,\cW)$ on $L^2(\bbR^2_\cW,\mu)$ is a strongly local regular Dirichlet form. Moreover, $(\cE,\cW)$ is recurrent and conservative.
\end{prop}
\begin{proof} The fact that $(\cE,\cW)$ is a strongly local Dirichlet form is trivial and we will omit the proof here. What we are mostly interested in showing is regularity.
 Consider the subalgebra $\cC$ of $C_0(\bbR^2_\cW)$ generated by $u\circ\pi$ for $u\in C_0^\infty(\bbR^2)$ and $\psi_0$.
 We shall prove that $\cC$ is a core for $(\cE,\cW)$ on $L^2(\bbR^2_\cW,\mu)$, that is, it is dense in $C_0(\bbR^2_\cW)$ with respect to the uniform topology and in $\cW$ with respect to $\cE_1$.

 First, $\cW = \cH + \bbR\psi_0$ and $C_0^\infty(\bbR^2)$ is dense in $\cH$, it follows that $\Pi(C_0^\infty(\bbR^2)+\bbR \psi_0) \subset \cC$ is dense in $\cW$ with respect to $\cE_1$ and thus $\cC$ is dense in $\cW$.
 Second, it is immediate to see that $\cC$ is a subalgebra of $C_0(\bbR^2_\cW)$ that separates points and vanishes nowhere.
 It follows by the Stone-Weierstrass Theorem that $\cC$ is dense $C_0(\bbR^2_\cW)$ with respect to the uniform topology.

 Recurrence and conservativeness can be proved by employing the same argument and test functions as in Proposition~\ref{prop:properties}.
\end{proof}

As a consequence of Proposition~\ref{prop:regularity} and of Theorem~\ref{thm:huntprocess}, there exists a Hunt process $({\{Y_t\}}_{t\geq 0}, P^\cW_x)$, $x\in\bbR^2_\cW$
which is uniquely determined up to a properly exceptional set and it is associated to the Dirichlet form $(\cE,\cW)$ on $L^2(\bbR^2_\cW,\mu)$.
As $(\cE,\cW)$ is strongly local, we can take $(Y,P_x^\cW)$ to have continuous sample paths for all $x\in \bbR^2_\cW$. Before introducing hitting times of the newly added points, we shall compute the capacity of $\cO$. We denote by $\capa_1^\cW$ the $1$-capacity with respect to $(\cE,\cW)$ on $L^2(\bbR^2_\cW,\mu)$.

\begin{prop} The followings hold:
 \begin{itemize}
  \item[i)] $\capa_1^\cW(\cO_{\Theta}) = 0$,
  \item[ii)] $\capa_1^{\cW}(\{ 0_+, 0_- \}) = \capa_1(\{ 0\})$,
  \item[iii)]  $\capa_1^{\cW}(\{ 0_+\}) = \capa_1^{\cW}(\{ 0_-\}) > 0$.
 \end{itemize}
\end{prop}
\begin{proof} We start by proving i). It follows from~\cite[Theorem 2.1.1]{fukushima2011dirichlet} that
 \[
  \capa_1^\cW((\pi/2, \pi)) = \sup_n \,\capa_1^\cW((\pi/2+1/n, \pi-1/n)).
 \]
 Thus, it suffices to show that for $\delta>0$ we have $\capa_1^\cW((\pi/2+\delta, \pi-\delta)) = 0$. Define
 \[
  \bar{A}_{\epsilon,\delta} \ldef A^+_{\epsilon,\delta} \cup A^-_{\epsilon,\delta}  \cup (\pi/2+\delta, \pi-\delta)
 \]
 with $A^+_{\epsilon,\delta}, A^-_{\epsilon,\delta}$ defined as in~\eqref{eq:cones}.
 It is easy to see that an inequality analogous to~\eqref{ineq:capacityofcones} holds true for $\bar{A}_{\epsilon,\delta}$. In fact, it suffices to consider the same test functions and extend them continuously to $\cO$. Therefore,
 \[
  \capa^\cW_1((\pi/2+\delta, \pi-\delta)) \leq \limsup_{\epsilon\to 0}\,\capa_1^\cW(\bar{A}_{\epsilon,\delta} ) = 0.
 \]
 In view of i), $\capa_1^\cW(\cO) = \capa_1^\cW(\{ 0_+,0,_-\})$. Thus, to prove ii) we just have to show that $\capa_1^{\cW}(\cO) = \capa_1(\{ 0\})$. This follows by the simple observation that there is a one to one correspondence between open neighborhoods of  $\{ 0\}$ in $\bbR^2$ and of $\cO$ in $\bbR^2_\cW$, via the map $\pi$.
 Moreover, by the very definition of capacity it is clear that
 \[
  \capa_1(A) = \capa_1^\cW(\pi^{-1}(A))
 \]
 for any neighborhood $A$ of $0$. Taking the infimum over all the open neighborhoods of $0$ leads to the conclusion. For iii), we observe first that $\capa^\cW_1(\{0_+\})=\capa^\cW_1(\{0_-\})$
 is just a consequence of the inner symmetry of the model. Moreover, thanks to ii)
 \[
  0<\capa_1(\{0\}) = \capa^\cW_1(\{0_+,0_-\}) \leq \capa^\cW_1(\{0_+\})+\capa^\cW_1(\{0_-\}),
 \]
 where the last inequality is a straightforward consequence of~\cite[Lemma 2.1.2]{fukushima2011dirichlet}. Finally, we get
 \[
  0 < \capa^\cW_1(\{0_+\})+\capa^\cW_1(\{0_-\}) = 2 \capa^\cW_1(\{0_+\})
 \]
 and the conclusion follows.
\end{proof}

Since $\cO_\Theta$ has zero capacity, it is exceptional. By~\cite[Theorem 4.4.1]{fukushima2011dirichlet} we can find a properly exceptional set $\cN$ such that $\cO_\Theta\subset \cN$, in particular
\begin{equation}\label{eq:properlyexc}
 P_x^\cW(Y_t\in \cO_\Theta, \mbox{ for some } t\geq 0)=0,\quad \forall x\in \bbR^2_\cW\setminus \cN.
\end{equation}
Let us define the following hitting times
\begin{equation}\label{eq:hittingtimes}
 \sigma_+\ldef\inf \{t>0 : Y_t = 0_+\},\quad \sigma_-\ldef \inf \{t>0 : Y_t = 0_-\},
\end{equation}
and
\begin{equation}\label{eq:hittingtimeO}
 \sigma \ldef \inf \{t>0 : Y_t \in \cO \}.
\end{equation}
Observe that~\eqref{eq:properlyexc} implies that for q.e.\ $x\in \bbR^2_\cW$ we have
\begin{equation}\label{noboundary}
 \sigma = \sigma_+\wedge \sigma_-\;,\quad P_x^\cW\mbox{-almost surely}.
\end{equation}
In analogy with the previous section we also define the hitting probabilities
\begin{equation}\label{eq:hitprobW}
 \varphi^+(x)\ldef P_x^\cW(\sigma_+ < \sigma_-),\quad   \varphi^-(x)\ldef P_x^\cW(\sigma_- < \sigma_+)
\end{equation}
and the $\alpha$-order hitting probabilities
\begin{equation}\label{eq:potW}
 u^+_\alpha(x)\ldef E_x^\cW\left[e^{\alpha \sigma}; \sigma = \sigma_+\right],\quad   u^-_\alpha(x)\ldef E_x^\cW\left[e^{\alpha \sigma}; \sigma = \sigma_-\right].
\end{equation}
As before, we can consider the part of the process $(Y,P^\cW_x)$ on $\bbR^2_\cW\setminus \{ 0_+, 0_-\}$, that is, the process $(Y^0,P^\cW_x)$, $x\in \bbR^2_\cW\setminus \{ 0_+, 0_-\}$ obtained from $Y$ by killing the sample paths upon hitting the set $\{ 0_+, 0_-\}$. Moreover, we can also consider the part  of $Y$ on $\bbR^2_0$ which we denote by $(Y^*, P^\cW_x)$.

\begin{remark}\label{rem:killed}
 In view of~\eqref{eq:properlyexc} one has
 \[
  P^\cW_x(Y^*_t = Y^0_t,\,\forall t\geq 0) = 1,\quad \mbox{q.e.\ }x\in \bbR^2_0.
 \]
 Moreover, we one can easily see that $(Y^*,P^\cW_x)$ is associated to the Dirichlet form $(\cE,\cH_0)$ on $L^2(\bbR^2_0,\mu)$, and thus it is equivalent to $(X^0,P_x)$, $x\in \bbR^2_0$ defined in the previous section.
\end{remark}
We denote by
\begin{equation*}
 p^\cW_t f(x) \ldef E^\cW_x[f(Y_t)],\quad G^\cW_\alpha f(x) \ldef E^\cW_x\left [\int_0^\infty e^{-\alpha t} f(Y_t)\,dt\right]
\end{equation*}
the transition function and the resolvent of $Y$. Similarly, we note by $p^{\cW,0}_t$ and $G^{\cW,0}_\alpha$ the same quantities for $Y^0$.

\begin{lemma}\label{lem:ident} There exists a properly exceptional set $\cN\subset \bbR^2_0$
 such that for all $f\in L^2(\bbR^2_0,\mu)$ all $\alpha>0$ and $t>0$
 \[
  G^0_\alpha f(x) = G^{\cW, 0}_\alpha f(x), \quad p^0_t f(x) = p^{\cW, 0}_t f(x),\qquad \forall x\in \bbR^2_0\setminus \cN.
 \]
\end{lemma}
\begin{proof} This follows directly from Remark~\ref{rem:killed}.
\end{proof}

In the next proposition we show that the projection of $(Y_t)_{t\geq 0}$ on $\bbR^2$ via $\pi$ has the property that it exits the origin from the same corner it entered. In particular, $(\pi(Y_t))_{t\geq 0}$ is not strongly Markov.

\begin{prop}  Fix $\epsilon \in (0,\pi/6)$ and define the subsets of $\bbR^2$
\[
C^+_\epsilon = (0,1)\times (-\epsilon,\pi/2+\epsilon) \cup\{0\}, \quad C^-_\epsilon = (0,1)\times (-\pi-\epsilon,-\pi/2+\epsilon)  \cup\{0\}.
\]
 Then, for q.e.\ $x\in\bbR^2_{0}$ and $P^\cW_x$-almost all $\omega$ there exists $\delta = \delta(\epsilon,\omega)$ such that either
 \begin{equation*}
\pi(Y_t)\in C^+_\epsilon,\,\forall t\in(\sigma-\delta,\sigma+\delta) \quad\text{or}\quad\pi(Y_t)\in C^-_\epsilon,\,\forall t\in(\sigma-\delta,\sigma+\delta).
 \end{equation*}
 where $\pi:\bbR^\cW\to\bbR^2$ was defined in~\eqref{projection}.
\end{prop}
\begin{proof} It is  immediate to see that $\pi^{-1}(C_\epsilon^+),\pi^{-1}(C_\epsilon^-)\subset \bbR^2_\cW$ are disjoint open  neighborhoods of $0_+$ and $0_-$ respectively.
As observed in~\eqref{noboundary}, for q.e.\ $x\in \bbR^2_\cW$ we have
\[
 \sigma = \sigma_+\wedge \sigma_-\;,\quad P_x^\cW\mbox{-almost surely}.
\]
Morever, $(Y,P_x^\cW)$ has continuous sample paths in $\bbR^2_\cW$ for quasi every point  $x\in \bbR^2_\cW$ because $(\cE,\cW)$ is a strongly local Dirichlet form.

Since  $Y_{\sigma} \in \{0_+,0_-\}$ $P_x^\cW\mbox{-almost surely}$ for q.e.\ $x\in\bbR^2_\cW$, the existence of $\delta$ as in the statement of the proposition follows from the continuity of the sample paths.
\end{proof}

\begin{prop}\label{prop:irreducibility} The Dirichlet form $(\cE,\cW)$ on $L^2(\bbR^2_\cW,\mu)$ is irreducible. In particular, $\varphi^+(x)+\varphi^-(x) = 1$, for q.e.\ $x\in \bbR^2_{\cW}$. Moreover, $\varphi^+(x),\varphi^-(x)>0$ for q.e.\ $x\in \bbR^2_0$.
\end{prop}
\begin{proof} Irreducibility follows in the same way as Proposition~\ref{prop:irreducible}, using the irreducibility of $(\cE,\cH_0)$ on $L^2(\bbR^2_0,\mu)$ and the fact that $\mu(\cO)=0$ by construction.
 From the irreducibility and~\cite[Theorem 4.7.1]{fukushima2011dirichlet} we get that for q.e.\ $x\in \bbR^2_\cW$
 \[
  \varphi^+(x)+\varphi^-(x) = P_x^\cW(\sigma<\infty) =1.
 \]

 We prove now the last part of the statement, we just show $\varphi^+>0$ q.e.\ as the other is similar.
 We notice that the part of $(Y,P_x^\cW)$ on $\bbR^\cW\setminus \{ 0_-\}$, which we call $W$, is still an irreducible process. Call $\tau\ldef\inf \{ t>0 : W_t = 0_+ \}$, then by~\cite[Theorem 4.7.1]{fukushima2011dirichlet}
 \[
  \varphi^+(x)=P^\cW_x(\sigma_+<\sigma_-) = P_x^\cW(\tau<\infty) > 0,\quad \mbox{q.e. }x\in \bbR^2_\cW\setminus \{0_-\}.
 \]
\end{proof}

Proposition~\ref{prop:irreducibility} tells us that we can apply the theory of many-point symmetric extensions~\cite[Theorem 7.7.3]{chen2012symmetric}. This allows to describe the process $(Y,P^\cW_x)$ via quantities that depend only on the process associated to $(\cE,\cH_0)$ on $L^2(\bbR^2_0,\mu)$. To present the theorem we shall need the following additional notation. Define
\begin{equation}\label{eq:energyfun}
 \gamma^{+-}\ldef \lim_{t\to 0}\frac{1}{t}\langle\varphi^+-p_t^0 \varphi^+,\varphi^-\rangle
\end{equation}
and
\begin{equation*}
 \gamma^{+-}_\alpha \ldef \alpha \langle u_\alpha^+,\varphi^-\rangle, \quad
 \gamma^{++}_\alpha \ldef \alpha \langle u_\alpha^+,\varphi^+\rangle.
\end{equation*}
Notice that the limit~\eqref{eq:energyfun} exists since $\varphi^+$ and $\varphi^-$ are excessive with respect to $p_t^0$, because $p_t^0 \varphi^+ \uparrow \varphi^+$ and $p_t^0 \varphi^- \uparrow \varphi^-$ when $t\downarrow 0$.
Moreover, $\gamma^{++}_\alpha,\gamma^{+-}_\alpha$ are well defined since $u^+_\alpha+u^-_\alpha = u_\alpha\in L^1(\bbR^2_\cW,\mu)$ by Theorem~\ref{thm:char}.
\begin{theorem}\label{thm:semigroupW}
 For any $\alpha$ and $g\in L^2(\bbR^2_\cW,\mu)$ let $\phi \ldef G^\cW_\alpha g|_{\{ 0_+, 0_-\}}$, then $\phi$ satisfies
 \begin{align*}
  \phi(0_+) = & \frac{\gamma^{++}_\alpha \langle u_\alpha^+,g\rangle - \gamma^{+-}_\alpha \langle u_\alpha^-,g\rangle + \gamma^{+-}\langle
   u_\alpha,g\rangle}{(\gamma_\alpha^{++}+\gamma_\alpha^{+-})(\gamma_\alpha^{++}-\gamma_\alpha^{+-}+2 \gamma^{+-})}                        \\
  \phi(0_-) = & \frac{\gamma^{++}_\alpha \langle u_\alpha^-,g\rangle - \gamma^{+-}_\alpha \langle u_\alpha^+,g\rangle + \gamma^{+-}\langle
   u_\alpha,g\rangle}{(\gamma_\alpha^{++}+\gamma_\alpha^{+-})(\gamma_\alpha^{++}-\gamma_\alpha^{+-}+2 \gamma^{+-})}.
 \end{align*}
 Furthermore, $G^\cW_\alpha g$ admits the representation
 \begin{equation}\label{eq:res_representation}
  G^\cW_\alpha g(x) = G^0_\alpha g(x) + u^+_\alpha(x) \phi(0_+) + u^-_\alpha(x) \phi(0_-), \quad \mbox{q.e.\ }x\in \bbR^2_0.
 \end{equation}
 In particular, $\cW = \cH_0 + \mathrm{span}(u^+_\alpha, u^-_\alpha)$.
\end{theorem}
\begin{proof}
 The result follows by an application of~\cite[Theorem 7.7.3]{chen2012symmetric} and Lemma~\ref{lem:ident}.
 The basic idea involves looking at the orthogonal decomposition of $\cW$ into the space of functions that are $0$ on $\{ 0_+, 0_-\}$ and those that are harmonic extensions in $\bbR^2_\cW\setminus \{ 0_+, 0_-\}$ of functions defined on  $\{ 0_+, 0_-\}$; in this case, the space of harmonic extensions is the two-dimensional space $ \mathrm{span}(u^+_\alpha, u^-_\alpha)$.
 This decomposition, together with the construction of the trace of the Dirichlet form $(\cE,\cW)$ on $\{ 0_+, 0_-\}$, is performed in~\cite[Theorem 7.5.4]{chen2012symmetric}.
\end{proof}

\begin{remark}In the state space $\bbR^2_\cW$ consider the domain $\cK = \cH_0 + \bbR \psi $. The process associated to $(\cE,\cK)$ is killed in $\{ 0_-\}$ and reflected back in $\{ 0_+\}$.
\end{remark}

\subsection{One the active reflected Dirichlet space of $\cH_0$}
Reflected Dirichlet spaces play an important role in describing the boundary behavior of symmetric Markov processes and have been introduced by Silverstein in the seminal paper~\cite{silverstein1974reflected} and further investigated by Chen in~\cite{chen1992reflected}.
We refer to the monograph~\cite{chen2012symmetric} for examples and precise definitions.

In this subsection we show that $(\cE,\cW)$ is the \emph{active reflected Dirichlet space} of $(\cE,\cH_0)$. In doing so we shall consider the more general situation where $\rho:\bbR^d\to \bbR$ satisfies only $\rho, \rho^{-1}\in L^1_{\mathrm{loc}}(\bbR^d)$ and
\begin{equation}
  0< \inf_{x\in K} \rho(x) \leq \sup_{x\in K} \rho(x) < \infty
\end{equation}
for any compact $K\subset \bbR^d\setminus\{0\}$.

We need to introduce the space $\cH_{0,\mathrm{loc}}$ of functions that are locally in $\cH_0$.
We say that $u\in \cH_{0,\mathrm{loc}}$ if for all relatively compact open sets $U\subset \bbR^d\setminus\{ 0\}$ there exists $v\in \cH_0$ such that $u=v$ almost surely on $U$.
Then, according to Theorem 6.2.13 of~\cite{chen2012symmetric} the active reflected Dirichlet space $(\cE^{\mathrm{ref}},\cF^{\mathrm{ref}}_a)$ of $(\cE,\cH_0)$ can be described as
\begin{equation}\label{eq:cFa}
  \cF^{\mathrm{ref}}_a = \bigg\{u\in L^2(\bbR^d,\rho dx): \tau_k u \in \cH_{0,\mathrm{loc}}, \,\forall k\geq 1,\, \sup_{k\geq 1}
  \int_{\bbR^d}\mu_{\langle \tau_k u\rangle}(dx)<\infty\bigg\},
\end{equation}
\begin{equation}\label{eq:cEref}
  \cE^{\mathrm{ref}}(u,u) = \lim_{k\to \infty} \frac{1}{2}  \int_{\bbR^d}  \mu_{\langle \tau_k u\rangle}(dx),
\end{equation}
where $\tau_k u = (-k)\vee u \wedge k$ and where  $\mu_{\langle v\rangle}$ is the energy measure associated to $v \in \cH_{0,\mathrm{loc}}$ with respect to  $(\cE,\cH_0)$ (see eq. (3.2.20) and the discussion at page 130 in~\cite{fukushima2011dirichlet}).

\begin{prop}\label{prop:reflected} We have $(\cE,\cW) = (\cE^{\mathrm{ref}},\cF^{\mathrm{ref}}_a)$.
\end{prop}
\begin{proof} We start with the observation that due Theorem 4.3.11 of~\cite{chen2012symmetric} we have for all bounded $v\in \cH_0$
\[
\mu_{\langle v\rangle}(dx) = 2 |\nabla v|^2 \rho \,dx,
\]
which easily extends to all bounded $v\in\cH_{0,\mathrm{loc}}$. In view of~\eqref{eq:cFa},~\eqref{eq:cEref} and dominated convergence we can rewrite
\[
\cF^{\mathrm{ref}}_a = \bigg\{u\in L^2(\bbR^d,\rho dx):  u \in \cH_{0,\mathrm{loc}},
\int_{\bbR^d} |\nabla u|^2 \rho dx<\infty\bigg\}
\]
and $\cE^\mathrm{ref}(u) = \int |\nabla u|^2 \rho dx$ for all $u\in \cF_a^\mathrm{ref}$. It remains to show that $\cW = \cF^{\mathrm{ref}}_a$.

We start with $\cW \subseteq \cF^{\mathrm{ref}}_a$. Let $u\in \cW$, then by definition $u\in L^2(\bbR^d,\rho dx)$  and $ \int |\nabla u|^2 \rho dx <\infty$.
Moreover, $u$ can be approximated on compact subsets of $\bbR^d\setminus \{0\}$ by smooth functions with respect to $\cE_1$. This is possible because $\rho,\rho^{-1}$ are locally bounded in $\bbR^d\setminus \{0\}$. Thus, $u\in \cH_{0,\mathrm{loc}}$ and consequently $u\in\cF^{\mathrm{ref}}_a$.

We proceed by showing $\cF^{\mathrm{ref}}_a\subseteq \cW$.
If $u\in \cF^{\mathrm{ref}}_a$, then by definition $u\in L^2(\bbR^d,\rho dx)$ and $\cE(u,u)<\infty$.
 We are left to show that $u\in W^{1,1}_{\mathrm{loc}}(\bbR^d)$. As $u\in \cH_{0,\mathrm{loc}}$ we know that $u$ is weakly differentiable on $\bbR^d\setminus \{0\}$.
Thus, $u$ is absolutely continuous along almost every line contained in $\bbR^d\setminus \{0\}$.
Moreover, $u,\partial_k u$ belong to $L^1_{\mathrm{loc}}(\bbR^d)$ for all $k=1,\ldots,d$ because $\cE_1(u,u)<\infty$ (see~\eqref{eq:localL1}). If we  integrate by parts along the $k$th direction, we see that for all $\phi \in C^\infty_0(\bbR^d)$ and all $k=1,\ldots, d$
\[
\int_{\bbR^d}  u \partial_k \phi dx = -\int_{\bbR^d}  \partial_k  u \phi dx.
\]
This shows that $u$ is weakly differentiable on $\bbR^d$ and thus $u\in W_{\mathrm{loc}}^{1,1}(\bbR^d)$.
\end{proof}

\begin{remark}
  According to Theorem 6.6.9 of~\cite{chen2012symmetric}, it follows from Proposition~\ref{prop:reflected} that $(\cE,\cW)$ is the maximal Silverstein extension of $(\cE,\cH_0)$.
\end{remark}

\section{Further examples}
\subsection{Higher rank extensions} In the previous section we have seen a weight on the plane which gave $\cH\neq \cW$ with $\cH$ having codimension one in $\cW$. A very naive modification of that example which does not involve the introduction of new points of degeneracy can also produce higher rank extensions.
Fix $a:[0,1]\to [0,\infty)$ satisfying~\eqref{ass:integrability},~\eqref{ass:onepointdeg} and~\eqref{ass:onerank} (e.g. $a(r) = r^\alpha$ with $0<\alpha<2$).
Let $N\geq 2$, we now split the plane into $2N$ cones of angle $\pi/N$ meeting at the origin. More precisely, we define in polar coordinates $(r,\theta)$
\[
 C_i \ldef (0,\infty)\times[\pi(i-1)/N,\pi i/N) \quad i=1,\dots 2N
\]
and set $\rho: \bbR^d \to [0,\infty)$ by
\begin{equation}\label{def:weight_two_higherrank}
 \rho(x)\;\ldef \;\left \{
 \begin{array}{ll}
  a{(|x|)}^{-1}\;, & \quad x\in C_i \cap B(0,1),\, i\mbox{ odd},  \\
  a(|x|)\;,        & \quad x\in C_i \cap B(0,1),\, i\mbox{ even}, \\
  1\;,             & \quad \mbox{otherwise}
 \end{array}
 \right.
\end{equation}
By the same techniques as in Section~\ref{sec:guidingmodel} it can be shown that
\[
 \cW = \cH + \mathrm{span}(\psi_1,\ldots \psi_N),
\]
where $\psi_i(x)\ldef{(1-|x|^2)}_+ \varphi_i(x)$ and
\begin{equation}\label{def:psi_i}
 \varphi_i(x)\;\ldef \;\left \{
 \begin{array}{ll}
  \cos\left(N\theta(x)/2\right) \;, & \quad x\in C_{2i-2},   \\
  1\;,                              & \quad x\in C_{2i-1},   \\
  \sin\left(N\theta(x)/2\right) \;, & \quad x\in C_{2i},     \\
  0\;,                              & \quad \mbox{otherwise}
 \end{array}
 \right.
\end{equation}
with the convention that $C_0\ldef C_{2N}$. This means that $\rho$ is not regular. A regularization of the Dirichlet form $(\cE,\cW)$ on $\bbR^2$ can be achieved by modifying the topology around the origin in such a way to make each $\psi_i$, $i\in \{ 1,\ldots, N\}$ continuous.
Roughly speaking, the new state space is obtained by splitting the origin in $N$ different points.
The stochastic process associated to the regularized Dirichlet form $(\cE,\cW)$ on this new state space will have the property of reaching the origin from one of the cones $C_{2i}$, $i\in \{ 1,\ldots, N\}$ and of being reflected back in the same cone from which it arrived.

\subsection{Higher dimensions} So far we discussed examples in dimension two, however the presence of non-regular weights is not a prerogative of the plane.
We shall now present another example due to Zhikov in~\cite[Section 5.2]{zhikov1998weighted}; the reader should keep in mind that a way to
construct such weights is to set up a situation where the process killed at the points of degeneracy approaches these points from ``disconnected'' regions.
\begin{figure}[h]
 \centering
 \includegraphics[width=0.5\textwidth]{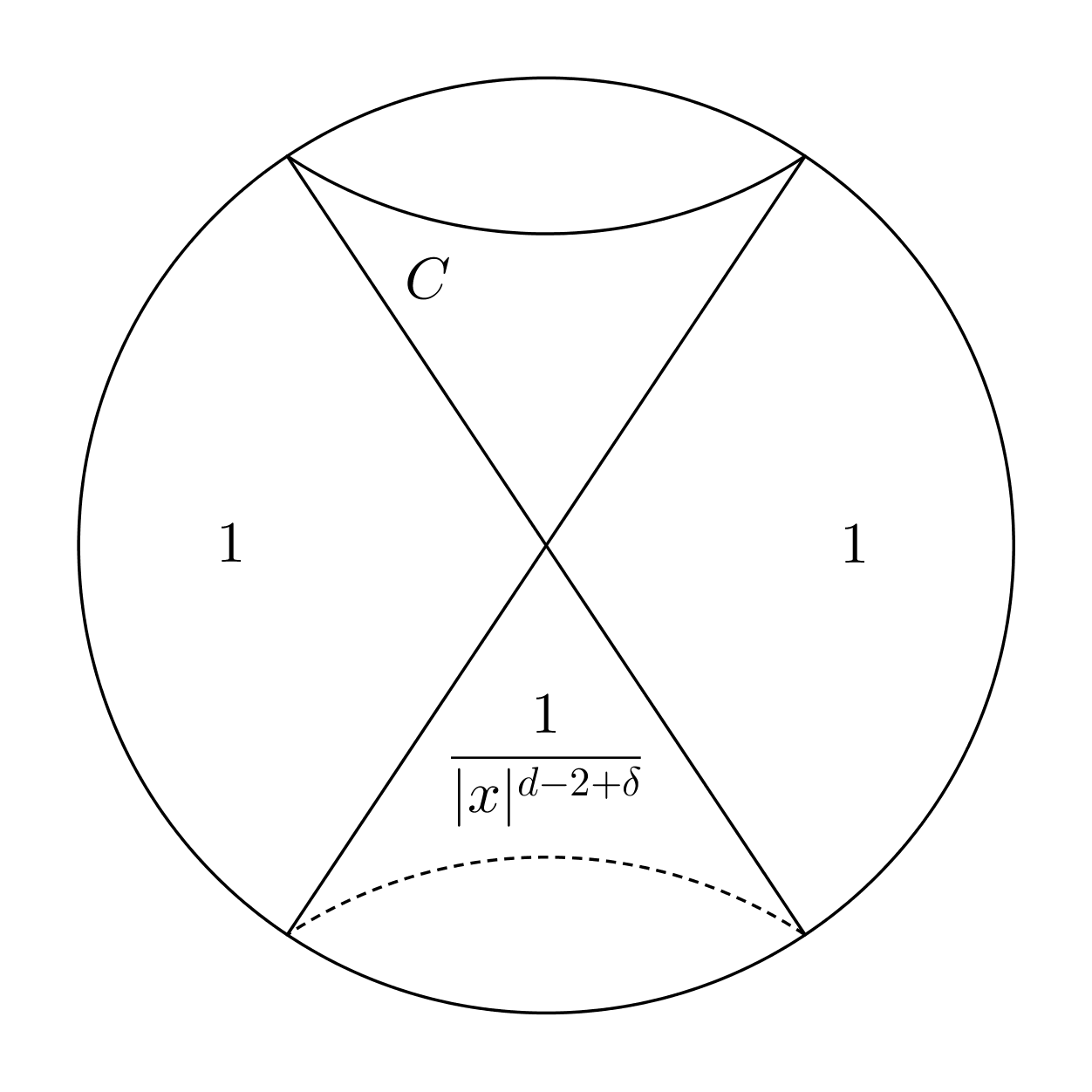}
 \caption{the weight $\rho(x)$.}\label{fig:cone}
\end{figure}

Let $d>2$ and $C$ be a circular cone, for $\delta>0$ small, define $\rho: \bbR^d \to [0,1]$ by
\begin{equation}\label{def:weight_geqthree}
 \rho(x)\;\ldef \;\left \{
 \begin{array}{ll}
  \frac{1}{|x|^{d-2+\delta}}\;, & \quad C \cap B(0,1),   \\
  1\;,                          & \quad \mbox{otherwise}
 \end{array}
 \right.
\end{equation}

It is proved in~\cite[Section 5.2]{zhikov1998weighted} that $\rho$ is not regular. This should not surprise the reader as the process associated to
\[
 L u = \frac{1}{\rho}\nabla\cdot(\rho \nabla u)
\]
has a drift pointing towards the origin inside $C$ and it is a Brownian motion otherwise. Thus we expect the process to hit zero from one of the two disconnected regions of $C\setminus \{ 0\}$.
By Proposition~\ref{prop:onepoint}, the origin will have positive capacity and
the same discussion as in Section~\ref{sec:H0H} can be carried out. In particular a regularization accounts in splitting the origin into two disconnected points. The process associated with the regularization will have the property that after hitting the origin, due to the continuity of the sample paths, it must depart from the same side it came.

\subsection{Fractal barrier} In this section we present an other example from~\cite{zhikov1998weighted} of a non-regular weight where a fractal barrier of positive capacity and zero measure is considered. This example was considered in~\cite{zhikov1996connectedness} to describe the phenomenon of fractal conductivity and homogenization for $\cH$- and $\cW$-solutions. We will keep the discussion at an informal level.
\begin{figure}[h]
 \centering
 \includegraphics[width=0.5\textwidth]{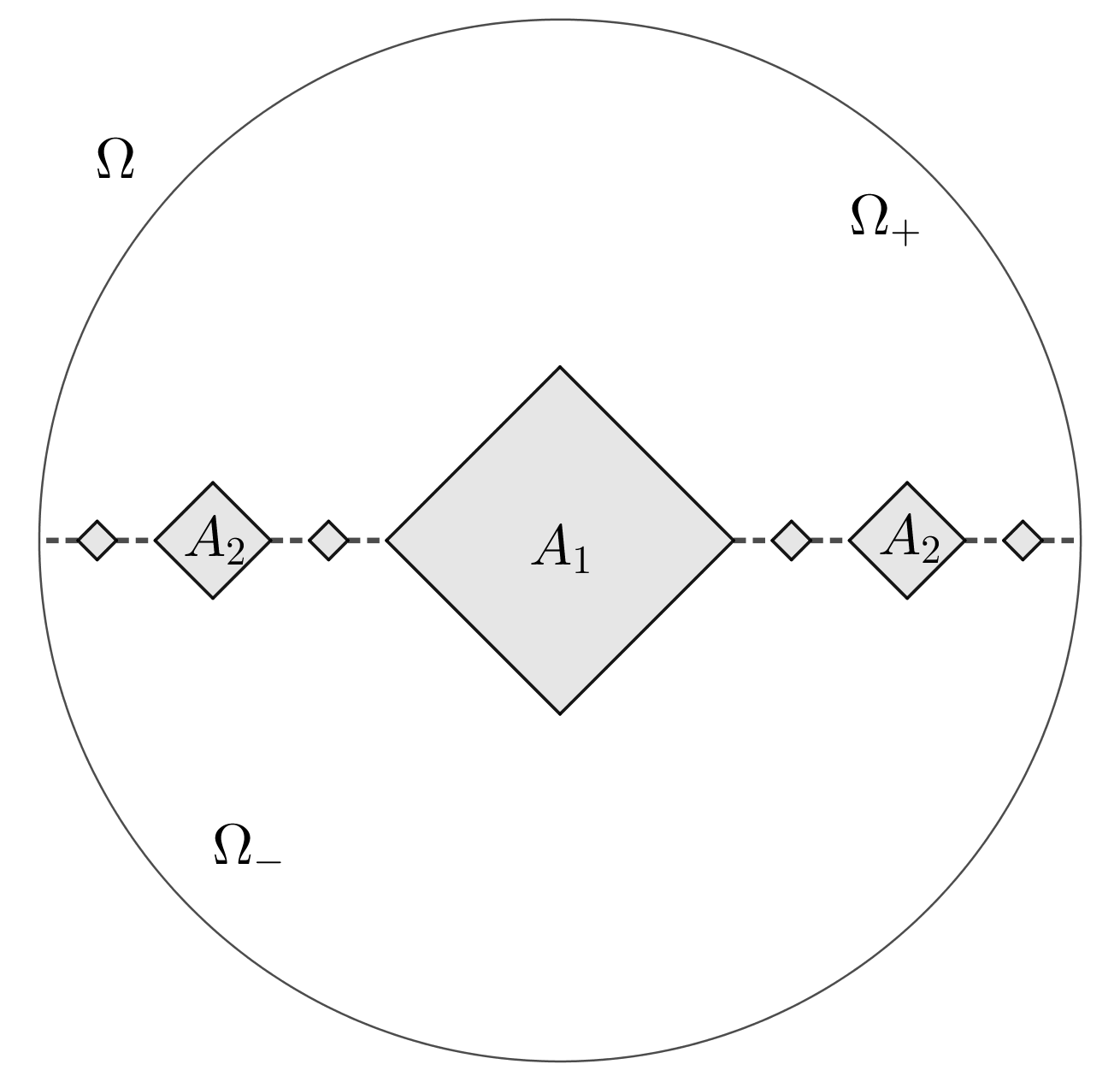}
 \caption{The Cantor barrier.}\label{fig:frac}
\end{figure}
We will construct a fractal barrier splitting a ball $\Omega$ of unitary diameter into two sides (see Figure~\ref{fig:frac}).
To this end, we consider the classical Cantor set on its diameter, that is, we divide the interval into three parts of the same length and delete the interior of the middle one.
We proceed by dividing the two remaining parts into three parts each and by deleting the middle parts, and so on. For $j\geq 1$, we denote by $A_j$ the union of the closed squares whose diagonals are the deleted intervals of length $3^{-j}$ (see Figure~\ref{fig:frac}).
The union of these squares and the diameter is a set $\Lambda$ which partitions $\Omega$ into two disjoint open domains $\Omega_+$ and $\Omega_-$ such that $\Sigma \ldef \overline{\Omega}_+\cap \overline{\Omega}_-$  is the Cantor set on the diameter. The weight $\rho:\Omega \to (0,\infty)$ has the following structure
\begin{equation}\label{def:weight_fract}
 \rho(x)\;\ldef \;\left \{
 \begin{array}{ll}
  a(x)\;, & \quad \mbox{on }\Lambda                \\
  1\;,    & \quad \mbox{on }\Omega\setminus\Lambda
 \end{array}
 \right.
\end{equation}
where $a(x) = 2^{-j}$ whenever $x\in A_j$. Clearly, $0<\rho\leq 1$ and $\rho^{-1} \in L^1(\Omega)$. In order to prove that $\cH\neq\cW$ the following test function is employed: $u:\Omega \to [0,1]$,
\[
 u(x) = 1,\quad x\in\Omega_+,\qquad u(x) = 0,\quad x\in\Omega_-,
\]
and $u|_{A_j}$ solves the Dirichlet problem on $A_j$ with boundary conditions one on $\overline{\Omega}_+\cap A_j$ and zero on $\overline{\Omega}_-\cap A_j$ for all $j\in \bbN$.

One can prove that $u\in \cW\setminus \cH$ (see~\cite[Section 5.3]{zhikov1998weighted}). This is due to the fact that $u$ is not constant on $\Omega\setminus \Lambda$, even though $\nabla u =0$ on $\Omega\setminus \Lambda$ almost everywhere. If $u$ were in $\cH$, then $\Omega\setminus \Lambda$ would not be $2$-connected, which is in contradiction with~\cite[Section 2]{zhikov1996connectedness} where $2$-connectedness is proven.

We want now understand the non-regularity of $\rho$ from the point of view of stochastic analysis, at least on a heuristic level.
We consider the diffusion process $(X,P_x)$ associated to $(\cE,\cH)$ on $L^2(\Omega,\mu)$ where $\mu = \rho dx$ as usual. Observe that it is a Brownian motion on $\Omega\setminus \Lambda$ and that $\Sigma$ has positive capacity because of Proposition~\ref{prop:regularity}.
This implies that $X$  hits $\Sigma$ with positive probability.
Immediately after hitting $\Sigma$ the process is going to continue its journey in either $\Omega_+$ or $\Omega_-$ with no preference due to the horizontal symmetry of $\Omega$.

Let us now consider the Dirichlet form $(\cE,\cW)$ on $L^2(\Omega,\mu)$. As smooth functions are not dense in $\cW$, $(\cE,\cW)$ is not regular and a regularization is called for. We notice that $u$ is continuous in  $\Omega\setminus \Sigma$ by construction and that the discontinuity arises in going through $\Sigma$ from $\Omega_+$ to $\Omega_-$.
This suggests that to regularize $(\cE,\cW)$ it could suffice to split the Cantor set $\Sigma$ into two disconnected copies so that, in the new state space, $\overline{\Omega}_+$ and $\overline{\Omega}_-$ would be disconnected.
In this case, $\Sigma$ would act as a \emph{hard barrier} and the only way for the process associated to $(\cE,\cW)$ to go through $\Lambda$ would be by  traversing one of the $A_j$.

This intuition leads to the following conclusion: the conductivity between $\Omega_+$ and $\Omega_-$ should be less for $(\cE,\cW)$ than for $(\cE,\cH)$, since for the process associated to $(\cE,\cW)$ it is harder to traverse $\Lambda$. This is what is proven rigorously in~\cite[Section 6]{zhikov1998weighted} in the context of homogenization.

\subsubsection*{Acknowledgment}
The authors would like to thank Professor M.\ Fukushima for his remarks and for showing interest in the paper.
The first author would like to thank the \emph{Laboratoire d'Excellence} LabEx Archim\`ede for supporting him financially in Marseille where this research was carried out.
Finally, but most importantly, the authors would like to thank  V.V.\ Zhikov as the present work was directly inspired by~\cite{zhikov1998weighted}.
Not long before Zhikov passed away, he was informed about this project and was kind enough to encourage us to complete it.
Vasily Vasilyevich's contribution to analysis is so rich that there is no doubt it will continue to be a source of inspiration for us and many generations of mathematicians.
\bibliographystyle{plain}
\bibliography{bibliography}

\begin{thebibliography}{10}

\bibitem{albeverio2003strong}
S.~Albeverio, Y.~Kondratiev, and M.~R{\"o}ckner.
\newblock Strong {F}eller properties for distorted brownian motion and
  applications to finite particle systems with singular interactions.
\newblock {\em Finite and Infinite Dimensional Analysis in Honor of Leonard
  Gross: AMS Special Session Analysis on Infinite Dimensional Spaces, January
  12-13, 2001, New Orleans, Louisiana}, 317:15, 2003.

\bibitem{cassano1994local}
F.~S. Cassano.
\newblock {\em On the local boundedness of certain solutions for a class of
  degenerate elliptic equations}.
\newblock Dipartimento di Matematica, Universita degli Studi di Trento, 1994.

\bibitem{chen2012symmetric}
Z.-Q. Chen and M.~Fukushima.
\newblock {\em Symmetric Markov Processes, Time Change, and Boundary Theory
  (LMS-35)}.
\newblock LMS Monographs. Princeton University Press, 2012.

\bibitem{chen2015one}
Z.-Q. Chen and M.~Fukushima.
\newblock One-point reflection.
\newblock {\em Stochastic Processes and their Applications}, 125(4):1368--1393,
  2015.

\bibitem{chen2005extending}
Z.-Q. Chen, M.~Fukushima, and J.~Ying.
\newblock Extending markov processes in weak duality by poisson point processes
  of excursions.
\newblock {\em Stochastic analysis and applications}, 2:153--196, 2005.

\bibitem{chen1992reflected}
Zhen-Qing Chen.
\newblock On reflected dirichlet spaces.
\newblock {\em Probability theory and related fields}, 94(2):135--162, 1992.

\bibitem{fukushima2011dirichlet}
M.~Fukushima, Y.~Oshima, and M.~Takeda.
\newblock {\em Dirichlet Forms and Symmetric Markov Processes}.
\newblock De Gruyter studies in mathematics. De Gruyter, 2011.

\bibitem{fukushima2005poisson}
M.~Fukushima and H.~Tanaka.
\newblock Poisson point processes attached to symmetric diffusions.
\newblock In {\em Annales de l'Institut Henri Poincare (B) Probability and
  Statistics}, volume~41, pages 419--459. Elsevier, 2005.

\bibitem{krylov2005strong}
N.~V. Krylov and M.~R{\"o}ckner.
\newblock Strong solutions of stochastic equations with singular time dependent
  drift.
\newblock {\em Probability theory and related fields}, 131(2):154--196, 2005.

\bibitem{ma2012introduction}
Z.-M. Ma and M.~R{\"o}ckner.
\newblock {\em Introduction to the theory of (non-symmetric) Dirichlet forms}.
\newblock Springer Science \& Business Media, 2012.

\bibitem{rockner2015non}
M.~R{\"o}ckner, J.~Shin, and G.~Trutnau.
\newblock Non-symmetric distorted brownian motion: strong solutions, strong
  {F}eller property and non-explosion results.
\newblock {\em arXiv preprint arXiv:1503.08273}, 2015.

\bibitem{silverstein1974reflected}
Martin~L Silverstein et~al.
\newblock The reflected dirichlet space.
\newblock {\em Illinois Journal of Mathematics}, 18(2):310--355, 1974.

\bibitem{zhikov1996connectedness}
V.~V. Zhikov.
\newblock Connectedness and homogenization. examples of fractal conductivity.
\newblock {\em Sbornik: Mathematics}, 187(8):1109, 1996.

\bibitem{zhikov1998weighted}
V.~V. Zhikov.
\newblock Weighted {S}obolev spaces.
\newblock {\em Sbornik: Mathematics}, 189(8):1139, 1998.

\bibitem{zhikov2011variational}
V.~V. Zhikov.
\newblock On variational problems and nonlinear elliptic equations with
  nonstandard growth conditions.
\newblock {\em Journal of Mathematical Sciences}, 173(5):463--570, 2011.

\bibitem{zhikov2013density}
V.~V. Zhikov.
\newblock Density of smooth functions in weighted {S}obolev spaces.
\newblock In {\em Doklady Mathematics}, volume~88, pages 669--673. Springer,
  2013.

\bibitem{zhikov2016density}
V.~V. Zhikov and M.~Surnachev.
\newblock On density of smooth functions in weighted {S}obolev spaces with
  variable exponents.
\newblock {\em St. Petersburg Mathematical Journal}, 27(3):415--436, 2016.

\end{thebibliography}

\end{document}